\documentclass[reqno,12pt]{amsart}
\usepackage{graphicx}
\usepackage{hyperref}
\usepackage[usenames,dvipsnames,svgnames,table]{xcolor}
\usepackage{pgf,ifpdf,graphics}

\usepackage{epstopdf}
\usepackage{amssymb,amsmath}
\usepackage{amsthm}
\usepackage{enumerate}
\newtheorem{theorem}{Theorem}[section]
\newtheorem{proposition}[theorem]{Proposition}
\newtheorem{lemma}[theorem]{Lemma}
\newtheorem{corollary}[theorem]{Corollary}
\newtheorem{remark}[theorem]{Remark}
\theoremstyle{definition}
\newtheorem{definition}[theorem]{Definition}
\newtheorem{example}[theorem]{Example}
\numberwithin{equation}{section}

\newcommand{\C}{{\mathbb C}}
\newcommand{\R}{{\mathbb R}}
\newcommand{\Z}{{\mathbb Z}}
\newcommand{\N}{{\mathbb N}}
\newcommand{\Q}{{\mathbb Q}}

\newcommand{\CL}{{\mathcal L}}
\newcommand{\CM}{{\mathcal M}}
\newcommand{\CP}{{\mathcal P}}

\newcommand{\alcove}{{\mathfrak{a}}}

\newcommand{\retroS}{{M}}
\newcommand{\polypp}{{\mathcal Q}}

\newcommand{\CR}{{\mathcal R}}
\newcommand{\CS}{{\mathcal S}}
\newcommand{\CV}{{\mathcal V}}
\newcommand{\la}{{\langle}}
\newcommand{\ra}{{\rangle}}
\newcommand{\e}{{\mathrm{e}}}

\newcommand{\codim}{\operatorname{codim}}

\newcommand{\lin}{\operatorname{lin}}

\newcommand{\vol}{\operatorname{vol}}
\def\binomial(#1,#2){\binom{#1}{#2}}
\def\mult(#1,#2){\binom{#1}{#2}}
\renewcommand{\a}{{\mathfrak{a}}}

\renewcommand{\b}{{\mathfrak b}}
\renewcommand{\c}{{\mathfrak{c}}}

\newcommand{\f}{{\mathfrak{f}}}
\renewcommand{\t}{{\mathfrak{t}}}
\newcommand{\s}{{\mathfrak{s}}}
\newcommand{\p}{{\mathfrak{p}}}
\newcommand{\q}{{\mathfrak{q}}}
\renewcommand{\u}{{\mathfrak{u}}}
\newcommand{\lattice}{\Lambda}

\makeatletter
\let\OurMathBbAux=\mathbb
\DeclareRobustCommand\OurMathBb{\OurMathBbAux}
\let\mathbb=\OurMathBb
\let\bfseries=\undefined
\DeclareRobustCommand\bfseries
{\not@math@alphabet\bfseries\mathbf
  \boldmath\fontseries\bfdefault\selectfont\let\OurMathBbAux=\mathbf}
\def\@thm#1#2#3{%
  \ifhmode\unskip\unskip\par\fi
  \normalfont
  \trivlist
  \let\thmheadnl\relax
  \let\thm@swap\@gobble
  \thm@notefont{\fontseries\mddefault\upshape\unboldmath}%   %%% <--- Added \unboldmath
  \thm@headpunct{.}% add period after heading
  \thm@headsep 5\p@ plus\p@ minus\p@\relax
  \thm@space@setup
  #1% style overrides
  \@topsep \thm@preskip               % used by thm head
  \@topsepadd \thm@postskip           % used by \@endparenv
  \def\@tempa{#2}\ifx\@empty\@tempa
    \def\@tempa{\@oparg{\@begintheorem{#3}{}}[]}%
  \else
    \refstepcounter{#2}%
    \def\@tempa{\@oparg{\@begintheorem{#3}{\csname the#2\endcsname}}[]}%
  \fi
  \@tempa
}
\makeatother

\newcommand{\tgreen}[1]{}
\newcommand{\tred}[1]{}

\title[Intermediate Sums on Polyhedra II]{Intermediate Sums on Polyhedra II:  Bidegree and Poisson formula}
\author{V. Baldoni}
\address{Velleda Baldoni: Dipartimento di Matematica, Universit\`a degli studi di  Roma ``Tor Vergata'',
Via della ricerca scientifica 1, I-00133, Italy}
\email{baldoni@mat.uniroma2.it}
\author{N. Berline}
\address{Nicole Berline:  \'Ecole Polytechnique, Centre de Math\'ematiques Laurent Schwartz, 91128 Palaiseau Cedex, France}
\email{berline@math.polytechnique.fr}
\author{J. A. De Loera}
\address{Jes\'us A. De Loera:  Department of
  Mathematics, University of California,
  Davis, One Shields Avenue, Davis, CA, 95616, USA}
\email{deloera@math.ucdavis.edu}
\author{M. K\"oppe}
\address{Matthias~K\"oppe:  Department of
  Mathematics, University of California,
  Davis, One Shields Avenue, Davis, CA, 95616, USA}
\email{mkoeppe@math.ucdavis.edu}
\author{M. Vergne}
\address{Mich\`ele Vergne: Institut de Math\'ematiques de
Jussieu -- Paris Rive Gauche, Batiment Sophie Germain, Case 7012, 
75205 Paris Cedex 13, France} \email{vergne@math.jussieu.fr}

\begin{document}
\maketitle{}

\begin{abstract}
  We continue our study of intermediate sums over polyhedra, interpolating
  between integrals and discrete sums, which were introduced by A.~Barvinok
  [\emph{Computing the {E}hrhart quasi-polynomial of a rational simplex},
  Math.~Comp.~\textbf{75} (2006), 1449--1466].  By well-known decompositions,
  it is sufficient to consider the case of affine cones $s+\c$, where $s$ is
  an arbitrary real vertex and $\c$ is a rational polyhedral cone.  For a
  given rational subspace~$L$, we integrate a given polynomial function~$h$
  over all lattice slices of the affine cone~$s+\c$ parallel to the subspace~$L$
  and sum up the integrals.  We study these intermediate sums by means 
  of the intermediate generating functions $S^L(s+\c)(\xi)$, 
  and expose the bidegree structure in parameters $s$ and $\xi$, which
  was implicitly used in the algorithms in our papers [\emph{Computation of
    the highest coefficients of weighted {E}hrhart quasi-polynomials of
    rational polyhedra}, Found.\ Comput.\ Math.\  \textbf{12} (2012), 435--469]
  and [\emph{Intermediate sums on polyhedra: Computation and real {E}hrhart
    theory}, Mathematika \textbf{59} (2013), 1--22].  The bidegree structure
  is key to a new proof for
  the Baldoni--Berline--Vergne approximation theorem for discrete generating
  functions [\emph{Local {E}uler--{M}aclaurin
  expansion of {B}arvinok valuations and {E}hrhart coefficients of rational
  polytopes}, Contemp.\ Math.\ \textbf{452} (2008), 15--33]% by patched intermediate generating functions
  , using the Fourier
  analysis with respect to the parameter~$s$ and a continuity argument.
  Our study also enables a forthcoming paper, in which we study intermediate sums
  over multi-parameter families of polytopes.
\end{abstract}

\tgreen{This is the version as in arXiv and submitted to journal,
  with new changes regarding degree terminology!  All changes are marked up in
  green. --Matthias 2014-08-30. Updated 2014-08-31 according to comments
  received from Michele and added more `degree' clarifications.}

% \clearpage
% \tableofcontents
%\clearpage
\section{Introduction}

Let $\p$ be a rational polytope in $V=\R^d$.  Computing the volume of the
polytope~$\p$ and counting the integer points in~$\p$ are two fundamental
problems in computational mathematics, both of which have a multitude of
applications.  The same is true for weighted versions of these problems.  Let
$h(x)$ be a polynomial function on $V$.  Then we consider the problems to compute the
integral 
$$I(\p, h) = \int_\p h(x)\, \mathrm dx$$
and the sum of values of $h(x)$ over the set of integer points of $\p$,
$$
S(\p,h)=\sum_ {x\in \p\cap \Z^d}h(x).
$$

\subsection{Intermediate sums}

The integral $I(\p,h)$ and the sum $S(\p,h)$ have an interesting common
generalization, the so-called \emph{intermediate} sums $S^L(\p,h)$, where $L\subseteq V$
is a rational vector subspace.  These sums
interpolate between the discrete sum $S(\p,h)$ and the integral $\int_\p
h(x)\,\mathrm dx$ as follows. For a polytope $\p\subset V$ and a polynomial
$h(x)$, we define
$$
S^L(\p,h)= \sum_{y} \int_{\p\cap (y+L)} h(x)\,\mathrm dx,
$$
where the summation variable $y$ runs over a certain projected lattice, so that
the polytope $\p$ is sliced along affine 
subspaces parallel to $L$ through lattice points and the integrals
of $h$ over the slices are added up. For $L=V$, there is only one
term in the sum, and $S^V(\p,h)$ is just the integral~$I(\p,h)$. 
For $L=\{0\}$, we recover the discrete sum $S(\p,h)$. 

Intermediate sums were introduced as a key tool in a remarkable construction
by Barvinok~\cite{barvinok-2006-ehrhart-quasipolynomial}.  Consider the one-parameter family of dilations
$t\p$ of a given polytope~$\p$ by positive integers~$t$, which is studied in
\emph{Ehrhart theory}.  A now-classic result is that the number $S(t\p, 1)$ is a quasi-polynomial function of
the parameter~$t$.\footnote{A quasi-polynomial takes the form of a polynomial
  whose coefficients are periodic functions of~$t$, rather than constants. 
  In traditional Ehrhart theory, only integer dilation factors~$t$ are considered, and so a
  coefficient function with period~$q$ can be given as a list of $q$~values,
  one for each residue class modulo~$q$.  However, the approach to
  computing Ehrhart quasi-polynomials via generating functions of parametric
  polyhedra leads to a natural, shorter representation of the coefficient functions as closed-form formulas (so-called
  step-polynomials) of the dilation parameter~$t$, using the 
  ``fractional part'' function.  These closed-form formulas are naturally valid for
  arbitrary non-negative \emph{real} (not just integer) dilation
  parameters~$t$.
  This fact was implicit in the computational works following this
  method \cite{Verdoolaege2007parametric,Verdoolaege2005PhD}, and was made
  explicit in \cite{koeppe-verdoolaege:parametric}.   
  The resulting real Ehrhart theory has recently caught the interest of other
  authors \cite{linke:rational-ehrhart,HenkLinke}.}
Its crudest asymptotics (the highest-degree
coefficient) is given by the volume $I(\p,1)$ of the polytope.  Barvinok's
construction in~\cite{barvinok-2006-ehrhart-quasipolynomial} provides efficiently computable \emph{refined asymptotics} for $S(t\p,
1)$ in the form of the highest $k$ coefficients of the
quasi-polynomial, where $k$ is a fixed number.  This
is done by computing certain \emph{patched sums}, which are finite linear
combinations $\sum_{L\in \CL} \rho(L) S^L(\p,1)$ of intermediate
sums.  In~\cite{so-called-paper-1} and \cite{so-called-paper-2}, we gave a refinement and
generalization of Barvinok's construction, based on the
Baldoni--Berline--Vergne approximation theorem for discrete generating
functions \cite{Baldoni-Berline-Vergne-2008}, in which we handle the general weighted case
and compute the periodic coefficients as closed-form formulas (so-called
step-polynomials) of the dilation parameter~$t$.  These formulas are
naturally valid for arbitrary non-negative \emph{real} (not just integer) dilation
parameters~$t$.

%%%%% THIS DOES NOT SEEM TO FIT IN THE INTRODUCTION. --Matthias
%% Second, intermediate generating functions also need to be efficiently computed
%% to obtain an approximation scheme for certain mixed-integer polynomial optimization
%% problems \cite{berline-koeppe-vergne:mipsum-ext-abstract}.

\subsection{Multi-parameter families of polyhedra and their intermediate generating functions}

In the present article, we continue our study.  Our ultimate goal, which will
be achieved in the forthcoming article \cite{3-polynomials}, is to study
intermediate sums for families
of polyhedra governed by several parameters.  This interest is
motivated in part by the important applications in compiler optimization and
automatic code parallelization, in which multiple parameters arise naturally (see
\cite{Clauss1998parametric,Verdoolaege2005PhD,Verdoolaege2007parametric} and
the references within).  However, as we will explain below, the parametric
viewpoint enables us to prove fundamental results about intermediate sums
that are of independent interest.

Similar to
\cite{koeppe-verdoolaege:parametric,beck:multidimensional-reciprocity,HenkLinke}, let
$\p(b)$ be a parametric semi-rational \cite{so-called-paper-2} polyhedron in
$V=\R^d$, defined by inequalities
$$
\p(b)=\bigl\{\, x\in V : \la \mu_j,x\ra\leq b_j,\; j=1,\ldots,N \,\bigr\},
$$
where  $\mu_1,\mu_2,\ldots,\mu_N$ are \emph{fixed} linear forms with integer coefficients, and
the parameter vector $b=(b_1,\ldots, b_N)$ varies in $\R^N$.  The study of the counting functions
$b \mapsto S(\p(b), 1)$ then includes the classical vector partition functions
\cite{Brion1997residue} as a special case.  Then the counting
function is a piecewise quasi-polynomial function of the parameter
vector~$b$.\footnote{Again this fact is well-known for the ``classical'' case, when $b$ runs in
  $\Z^N$. That it holds as well for arbitrary \emph{real}
  parameters~$b$ follows from the computational works using the
  method of parametric generating functions
  \cite{Verdoolaege2007parametric,Verdoolaege2005PhD}; it is made
  explicit in \cite{koeppe-verdoolaege:parametric}.}
We can extend this result to the general weighted intermediate case in the forthcoming
paper~\cite{3-polynomials}. 

As in the well-known discrete case ($L=0$), it is a powerful method to take this study
to the level of generating functions.  We consider the
\textit{intermediate generating function}
\begin{equation}
S^L(\p(b))(\xi)= \sum_{y} \int_{\p\cap (y+L)} \e^{\la \xi,x\ra}\,\mathrm dx,
\end{equation}
for $\xi\in V^*$, where the summation variable $y$ again runs over a certain
projected lattice.  

If $\p(b)$ is a polytope, the  function
$S^L(\p(b))(\xi)$ is a holomorphic function of $\xi$. 
Then it is not hard to see that, once the generating
function~$S^L(\p(b))(\xi)$ is computed, the value $S^L(\p(b),h)$ for any polynomial
function~$h(x)$ can be extracted \cite{so-called-paper-1,so-called-paper-2}.  
Indeed, if $\xi\in V^*$,
then $h(x)=\frac{\la \xi,x\ra^m}{m!}$  is a polynomial function on $V$, homogeneous of degree $m$. Then
the homogeneous component
$ S^L(\p(b))_{[m]}(\xi) $ of the holomorphic
function~$S^L(\p(b))(\xi)$ is equal to the intermediate sum
$S^L(\p(b),h)$.
For a general polynomial function $h(x)$, the result then follows by well-known
decompositions as sums of powers of linear
forms~\cite{baldoni-berline-deloera-koeppe-vergne:integration}.

If $\p(b)$ is a polyhedron, not necessarily compact, the generating
function $S^L(\p(b))(\xi)$ still makes sense as a meromorphic function of~$\xi$
with hyperplane singularities (near $\xi=0$),
that is, the quotient of a function which is holomorphic near $\xi=0$
divided by a finite product of linear forms.  
Then, a well-known decomposition of space allows us to write the generating
function $S^L(\p(b))(\xi)$ of the polyhedron $\p(b)$ as the sum of the
generating functions of the affine cones $s+\c$ at the vertices (Brion's
theorem).
Note that within a chamber, $s$ will be an affine linear function of~$b$. 
A crucial observation is that
for such meromorphic functions, homogeneous components are still
well-defined, and the operation of taking homogeneous components commutes with
Brion's decomposition and other decompositions.

Brion's decomposition allows us to defer the discussion of the piecewise
structure of $S^L(\p(b))$ corresponding to the chamber decomposition of the
parameter domain to the forthcoming paper~\cite{3-polynomials}.  In the
present paper, we study the more fundamental question of the dependence of the
generating function $S^L(s+\c)(\xi)$ of an affine cone~$s+\c$ and its
homogeneous components $S^L(s+\c)_{[m]}(\xi)$ on the apex~$s$ and the dual
vector~$\xi$, i.e., as functions of the pair $(s,\xi)\in V\times V^*$.

%% Emphasize the important multi-parameter case.  But careful about in which
%% setting (homogeneous parameters?) we can get similar approximation results.

%% The results stated here assume that $V$ is equipped with a lattice. In fact the rational structure can be  only defined on $V/L$???

%% Let us go back to $V=\R^d$. The first step is to reduce to the  extreme cases $L=\{0\}$ and $L=V$.
%% A precious technical tool makes it possible: the Brion--Vergne decomposition
%% of the indicator function of a  simplicial cone $\c$ as a signed sum of indicator functions of
%% cones which have  $L$ as a face, modulo the space spanned by indicator functions of cones
%% which contain lines \cite{Brion1997residue,so-called-paper-2}. When $L$ is a face of $\c$, the intermediate generating function
%% $S^L(\c)$ breaks up as the product of the purely discrete one on the quotient $V/L$
%% and the purely continuous one for the intersection
%% $V\cap L$.

\subsection{First contribution:  Bidegree structure of $S^L(s+\c)(\xi)$ in parameters $s$ and $\xi$}

Instead of the intermediate generating function $S^L(s+\c)(\xi)$, it is
convenient to study the shifted function
$$
M^L(s,\c)(\xi)= \e^{-\langle\xi,s\rangle}S^L(s+\c)(\xi);
$$
it already appears implicitly in the algorithms in our papers \cite{so-called-paper-1,so-called-paper-2}.
This function depends only on $s$ modulo $\Z^d+L$, in other words it is a function on $V/L$
which is periodic with respect to the projected lattice.

To illustrate the main features of this function, let us first
describe the dimension one case with $V =\R$,  $L=\{0\}$ and $\c=\R_{\geq 0}$.
We denote by $\{t\}$  the fractional part of a real number $t$,
defined by $\{t\}\in [0,1[$ and $t-\{t\}\in \Z$.
Then
\allowdisplaybreaks
\begin{align*}
  S(s+\c)(\xi)&=\sum_{n\geq \lceil s \rceil}\e^{n\xi} =\frac{\e^{\lceil s
      \rceil\xi}}{1-\e^{\xi}}=\frac{\e^{(s+\{-s\})\xi}}{1-\e^{\xi}},\\
  M(s,\c)(\xi)&=\e^{-s\xi}S(s+\c)(\xi)=\frac{\e^{\{-s\}\xi}}{1-\e^{\xi}}.
\end{align*}
The function $M(s,\c)(\xi)$ admits a decomposition into homogeneous
components, which in this example are given by the  Bernoulli polynomials:
$$
M(s,\c)_{[m]}(\xi)= - \frac{B_{m+1}(\{-s\})}{(m+1)!}\xi^m.
$$
Thus, as a function of $s\in \R$, $M(s,\c)_{[m]}(\xi)$ is a polynomial in
$\{-s\}$ of degree $m + 1$. (We will prove that, in general, it will be of
degree $m + d$, where $d$ is the dimension of~$V$; so the $\xi$-degrees and
the $s$-degrees are linked.)
Hence it is periodic (with period $1$),
and it coincides with a polynomial on each semi-open interval $]n,n+1]$. In particular, it is left-continuous.

To describe this structure in general, 
we introduce  the algebras of step-poly\-nom\-ials and quasi-polynomials on $V$.
A (rational) step-poly\-nomial is an element of  the algebra of functions on $V$ generated by the functions
$s\mapsto \{\langle\lambda,s\rangle\}$, where $\lambda\in \Q^d$. If
$q\lambda\in\Z^d$, then $s\mapsto \{\langle\lambda,s\rangle\}$ is a function
periodic modulo $q\Z^d$.
A quasi-polynomial is an element of the algebra of functions on $V$
generated by step-polynomials and ordinary polynomials.
Thus a step-polynomial is periodic with respect to some common multiple $q\Z^d$,
and a quasi-polynomial $f(s)$ is piecewise polynomial in the  sense that it restricts
to a polynomial function on any alcove associated with the
rational linear forms $\lambda\in \Q^d$ entering in its coefficients. (Alcoves are open polyhedral subsets of $V$
defined in Definition \ref{def:Psi-alcove}).

%%%%% TOO MUCH DETAIL - removed
%% We introduce subalgebras of step-polynomials and quasi-polynomials associated with
%%  $\c$ and $L$, more precisely with certain decompositions of~$\c$%  the Brion--Vergne decomposition of $\c$
%%  % and the decompositions of projected cones  into unimodular ones
%%  .

Our first result is to prove  that the homogeneous components
$M^L(s+\c)_{[m]}(\xi)$ are such step-polynomial functions of $s$, and we
compute their degrees as step-polynomials.
It turns out that the degree as step-polynomial functions of $s$  and the homogeneous degree in $\xi$
are linked.
This \emph{bidegree structure} gives a blueprint for constructing
algorithms, based on series expansions in a constant number of variables, 
that extract refined asymptotics from the generating function.  (We introduced
such algorithms in \cite{so-called-paper-1,so-called-paper-2} to compute
coefficients of Ehrhart quasi-polynomials; algorithms for general parametric
polyhedra will appear in the forthcoming paper~\cite{3-polynomials}.)

Furthermore, we prove that these homogeneous component functions enjoy the following property of
one-sided continuity
(Proposition~\ref{prop:alcoves}):
\emph{Let $s\in V$. For any $v\in L-\c$, we have
 $$
 \lim_{\substack{t\to 0\\t>0}}\retroS^L(s+tv, \c )_{[m]}(\xi)= \retroS^L(s, \c )_{[m]}(\xi).
 $$}

\subsection{Second contribution:  A new, Fourier-theoretic proof of the
  Baldoni--Berline--Vergne approximation theorem}

Our second main result concerns the  Fourier series of the periodic function
$s\mapsto M^L(s+\c)(\xi)$, in the sense of periodic $L^2$-functions 
with values in the space of meromorphic functions of $\xi$.
The Fourier coefficient at $\gamma\in \Z^d$ is easy to compute:
it is $0$ if $\gamma$ is not orthogonal to $L$,
 and otherwise it is the meromorphic continuation
 of the integral
$$
I(\c)(\xi + 2i \pi \gamma)=\int_{\c} \e^{\langle \xi + 2i \pi \gamma,x\rangle}\,\mathrm dx.
$$
We prove that $s\mapsto M^L(s+\c)(\xi)$  is the sum of its Fourier series, in the above sense.
As a corollary, Theorem~\ref{th:poisson}, we obtain the Fourier expansion
(Poisson formula) of the homogeneous component
\begin{equation}
\retroS^L(s,\c)_{[m]}(\xi)= \sum_{\gamma \in \lattice^*\cap L^\perp}
   \e^{\langle 2 i \pi \gamma,s\rangle}
\bigl(I(\c)(\xi+2i \pi \gamma)\bigr)_{[m]}.\label{eq:poisson-intro}
\end{equation}
For instance, in the dimension one case,
we recover the well-known Fourier series of the Bernoulli polynomials,  for $ k\geq 1$,
$$
\frac{B_{k}(\{s\})}{k!}= -\sum_{n\neq 0}\frac{\e^{2i\pi ns}}{(2i\pi n)^k}.
$$
%%%
We also determine the
poles and residues of $S^L(s+\c)(\xi)$ (Proposition~\ref{prop:residueSL}).  This will be of importance for the
forthcoming paper \cite{3-polynomials}.\smallbreak

We then give a new  proof of the Baldoni--Berline--Vergne theorem \cite{Baldoni-Berline-Vergne-2008}
on approximating the discrete generating function  $S(s+\c)$ of an affine  cone by  a linear combination of  functions $S^L(s+\c)$. 
We use  Fourier analysis in this proof, as did Barvinok
\cite{barvinok-2006-ehrhart-quasipolynomial} in his proof of his theorem
regarding the highest $k+1$ coefficients of Ehrhart quasi-polynomials. We sketch the crucial idea.

Let $0\leq k\leq d$ and let $\CL$ be a family of subspaces $L$ of $V$ which contains
the faces of codimension $\leq k$ of $\c$ and is closed under sum. Consider
the subset $\bigcup_{L\in \CL} L^{\perp}$ of $V^*$ and write its indicator
function  as a linear combination of the indicator functions of the spaces
$L^{\perp}$:$$ 
\Bigl[\bigcup_{L\in \CL} L^{\perp}\Bigr]=\sum_{L\in \CL}\rho(L)[L^{\perp}].
$$

We define a function that we call \emph{Barvinok's patched generating function}  by
$$S^{\CL}(s+\c)(\xi)=\sum_{L\in\CL}\rho(L) S^L(s+\c)(\xi).$$
Let us compute the difference
$S(s+\c)(\xi)-S^{\CL}(s+\c)(\xi)$.
By the Poisson formula%%% related to \eqref{eq:poisson-intro}
, we have (in the sense of local $L^2$-functions of~$s$)
$$
S(s+\c)(\xi)= \sum_{\gamma \in \lattice^*}I(s+\c)(\xi+2i \pi \gamma)$$
whereas for each of the terms corresponding to $L \in \CL$,
$$S^L(s+\c)(\xi)= \sum_{\gamma \in \lattice^*\cap L^\perp}
  I(s+\c)(\xi+2i \pi \gamma).$$
As $[\bigcup_{L\in \CL}L^{\perp}]=\sum_{L \in \CL} \rho(L)[L^{\perp}]$,
we obtain
\begin{equation}
  S(s+\c)(\xi)-S^{\CL}(s+\c)(\xi)=\sum_{ \substack {\gamma \in \Lambda^*\\ \gamma \notin
      \bigcup_{L\in\CL} L^\perp}}
  I(s+\c)(\xi+2i \pi \gamma).\label{eq:difference-s-patched-intro}
\end{equation}
The  poles of the function $\xi\mapsto  I(s+\c)(\xi+2i \pi \gamma)$  are  on a
collection of affine hyperplanes depending on the position of $\gamma$. 
For $\gamma$ outside of $\bigcup_{L\in\CL} L^\perp$, which is what we sum over
in~\eqref{eq:difference-s-patched-intro}, the homogeneous components of
$\xi$-degree $\leq -d+k$ vanish (Proposition \ref{prop:good-gamma}). 

As a consequence of the results on the bidegree structure,
Equation~\eqref{eq:difference-s-patched-intro} actually holds in the sense of
local $L^2$-functions of~$s$, separately for each of
the homogeneous components in $\xi$-degree.  Using the one-sided continuity
results, it follows that it actually holds as a pointwise result for all~$s$.

This yields a straightforward proof  of the
Baldoni--Berline--Vergne approximation theorem (Theorem~\ref{th:better-than-Barvinok}), i.e., the fact that
\emph{the functions $S(s+\c)(\xi)$ and  $ S^{\CL}(s+\c)(\xi)$
have the same homogeneous components  in $\xi$-degree $\leq -d+k$}.
We showed in  \cite{so-called-paper-1} that this approximation theorem implies 
generalizations of Barvinok's results in
\cite{barvinok-2006-ehrhart-quasipolynomial} on  the highest $k+1$
coefficients of Ehrhart quasi-polynomials.  In the forthcoming paper
\cite{3-polynomials}, we extend these results to the case of parametric polytopes.

% As a second consequence of these results, we prove
%  the surprising analyticity of the  sum, over the vertices  of a simple polytope,
%  of the minimal Barvinok generating functions of the cones at vertices,
%  Proposition \ref{prop:cone-by-cone-is-analytic}. \texttt{Dire qq chose sur intermediate Todd classes}

%\clearpage
\section{Intermediate generating functions $S^L_\lattice(\p)(\xi)$}

\subsection{Notations} \label{notation}
In this paper,  $V$ is a \emph{rational vector space}  of dimension $d$, that is
to say $V$ is a finite-dimensional real vector space with a lattice denoted
by  $\lattice$.   We will need to consider subspaces and quotient
spaces of $V$, this is why we cannot simply let $V=\R ^d$ and
$\lattice = \Z^d$.
A subspace $L$ of $V$ is called
\emph{rational} if $L\cap \lattice $ is a lattice in $L$. If $L$ is a
rational subspace, $V/L$ is also a rational vector space.
Its  lattice, the image of $\lattice$ in $V/L$,
is called the \emph{projected lattice} and denoted by $\lattice_{V/L}$.
A rational space $V$, with lattice $\lattice$, has a canonical
Lebesgue measure  for which  $V/\lattice$ has measure~$1$, denoted by $\mathrm dm_\lattice(x)$, or simply $\mathrm dx$.

A point $s\in V$ is called \emph{rational} if there exists $q\in \Z$, $q\neq 0$,  such that $q s\in \lattice$.
A \emph{rational affine subspace} is a rational  subspace  shifted by a rational element $s\in V$.
A \emph{semi-rational affine subspace} is a rational  subspace  shifted by any  element $s\in V$.

In this article, a \emph{(convex) rational polyhedron} $\p\subseteq V$ is  the intersection of a finite number of
closed halfspaces bounded by  rational hyperplanes,
a \emph{(convex) semi-rational polyhedron}  is  the intersection of a finite number of
closed halfspaces bounded by  semi-rational hyperplanes. The word convex will be implicit.
  For instance, if $\p\subset V $ is a
rational polyhedron, $t$~is a real number and $s$~is any point in~$V$, then
the dilated polyhedron $t\p$ and the shifted polyhedron
$s+\p$ are semi-rational. All polyhedra will be semi-rational in this paper.
When a stronger assumption is needed, it will be stated explicitely.

In this article, a \emph{cone} is a convex polyhedral rational  cone (with vertex $0$) and
an \emph{affine  cone} is the shifted set $s+\c$ of a  rational cone $\c$ by any $s\in V$.
 A  cone
$\c$ is called \emph{simplicial} if it is  generated by linearly independent
elements of $V $. A simplicial cone~$\c$ is called \emph{unimodular} if it
is generated by independent lattice vectors $v_1,\dots, v_k$ such
that $\{v_1,\dots, v_k\}$ is part of a basis of
$\lattice$. An affine cone~$\a$ is called simplicial (respectively,
unimodular) if the associated cone is.
An (affine) cone is called \emph{pointed} if it does not contain a line.

 A \emph{polytope} $\p$ is a  compact polyhedron.
 The set of vertices of $\p$ is denoted by $\CV(\p)$.
 For each vertex $s$, the cone of feasible directions at $s$  is denoted by
$\c_s$.

The dual vector space of $V$ is  denoted by $V^*$.
If $L$ is a subspace of $V$, we denote by $L^{\perp}\subset V^*$ the space of linear forms $\xi\in V^*$
which vanish on $L$.
The dual lattice of $\lattice$ is denoted by $\lattice^*\subset V^*$.
Thus $\langle\lattice^*,\lattice\rangle \subseteq \Z$.

 The indicator function of a subset  $E$ is  denoted by $[E]$.

\subsection{Basic properties of intermediate generating functions}
If $\p$ is a polytope in the vector space $V$, its generating function
$S_\lattice(\p)(\xi)$ defined by $\sum_{x\in \p\cap \lattice} \e^{\langle \xi,x\rangle}$
is a holomorphic function on the complexified dual space $V^*_\C$.
This is the reason why we consider functions on the dual space $V^*$ in the following definition.\smallskip

\begin{definition}~
  \begin{enumerate}[\rm(a)]
  \item We denote by $\CM_{\ell}(V^*)$ the ring of meromorphic functions
    around $0\in V^*_\C$ which can be written as a quotient
    $\frac{\phi(\xi)}{\prod_{j=1}^{N}  \la\xi,w_j\ra}$, where $\phi(\xi)$ is holomorphic
    near $0$ and $w_j$ are non-zero elements of $V$ in finite number.
  \item
    We denote by ${\mathcal R}_{[\geq m]}(V^*)$ the space of rational
    functions which can be written as $\frac{P(\xi)}{\prod_{j=1}^{N}
      \la\xi,w_j\ra}$, where $P$ is a homogeneous polynomial of degree greater
    or equal to $m+N$.  These rational functions are said to be
    \emph{homogeneous of degree at least~$m$}.
  \item We denote by ${\mathcal R}_{[ m]}(V^*)$ the space of rational
    functions which can be written as $\frac{P(\xi)}{\prod_{j=1}^{N}
      \la\xi,w_j\ra}$, where $P$ is homogeneous of degree $m+N$.
    These rational functions are said to be \emph{homogeneous of degree~$m$}.
  \end{enumerate}
\end{definition}
\tgreen{Added some wording from 3polys to update above definition.}

A function in $\CR_{[m]}(V^*)$ need not be  a polynomial, even if $m\geq0$.    For instance,
$\xi\mapsto {\frac{\xi_1}{\xi_2}}$ is homogeneous of degree $0$.

\begin{definition}
  For $\phi\in \CM_{\ell}(V^*)$, not necessarily a rational function, the
  \emph{homogeneous component} $ \phi_{[m]}$ of degree~$m$ of~$\phi$ is
  defined by considering $\phi(\tau \xi)$ as a meromorphic function of one
  variable~$\tau\in\C$, with Laurent series expansion
 $$
 \phi(\tau \xi) = \sum_{m\geq m_0} \tau^m \phi_{[m]}(\xi).
 $$
\end{definition}
Thus $\phi_{[m]}\in {\mathcal R}_{[m]}(V^*)$.

The intermediate generating functions of polyhedra which we study in  this article are elements of $\CM_{\ell}(V^*)$ which enjoy the following valuation property.
\begin{definition}
    An  $\CM_{\ell}(V^*)$-valued \emph{valuation} on the set  of semi-rational
    polyhedra $ \p\subseteq V$
 is a  map $F$
from this  set to the vector space  $\CM_{\ell}(V^*)$
such that whenever the indicator functions $[\p_i]$ of a
family of polyhedra $\p_i$ satisfy a linear relation $\sum_i r_i
[\p_i]=0$, then the elements $F(\p_i)$ satisfy the same
relation
$$
\sum_i r_i F(\p_i)=0.
$$
\end{definition}

Recall some of the definitions we introduced in \cite{so-called-paper-2}.
\begin{proposition}\label{prop:valuationSL}
Let $L\subseteq V$ be a rational subspace. There exists a unique
valuation  which associates a meromorphic function $S^L_\lattice(\p)\in \CM_{\ell}(V^*)$
to every semi-rational polyhedron $\p\subseteq
V$, so that
the following properties hold:

\begin{enumerate}[\rm(i)]
\item  If $\p$ contains a line, then $S^L_\lattice(\p)=0$.

\item
  \begin{equation}\label{eq:SL}
    S^L_\lattice(\p)(\xi)= \sum_{y\in \lattice_{V/L}} \int_{\p\cap (y+L)} \e^{\la
      \xi,x\ra}\,\mathrm dx,
  \end{equation}
  for every $\xi\in V^*$ such that the above sum converges.
\end{enumerate}
Here, $\lattice_{V/L}$ is the  projected  lattice
and $\mathrm dx$ is the Lebesgue measure on $y+L$ defined by the intersection lattice $L\cap\lattice$.
\end{proposition}

\begin{definition}
    $S^L_\lattice(\p)(\xi)$ is called the  \emph{intermediate  generating function} of $\p$
(associated to the subspace $L$).
\end{definition}
When there is no risk of confusion, we will  drop the subscript $\lattice$.

 The intermediate generating function
interpolates between the integral
$$
I_\lattice(\p)(\xi)= \int_\p \e^{\la \xi,x\ra} \,\mathrm dx
$$
which corresponds to $L=V$,  and the discrete sum
$$
S_\lattice(\p)(\xi)= \sum_{x\in \p\cap \lattice}\e^{\la \xi,x\ra}
$$
which corresponds to $L=\{0\}$.

 Formula \eqref{eq:SL} does not hold around $\xi=0$ when $\p$ is not compact. Near  $\xi=0$,
 $S^L_\lattice(\p)(\xi)$ has to be defined  by analytic continuation.
 For instance in dimension one, with $L=V=\R $, the integral
 $ \int_0^\infty \e^{\xi x } \,\mathrm dx = -\frac{1}{\xi}$
 converges only for $\xi<0$.
 Similarly the discrete sum $\sum_{n\geq 0}\e^{\xi n}=\frac{1}{1-\e^{\xi}}$ converges only for $\xi<0$.

Let us give the simplest example of the
 meromorphic functions   $S^L_\lattice(\p)(\xi)$ so obtained
when $V=\R$ is one-dimensional.
The formulae  are written in terms of  the \emph{fractional part} $\{t\}\in [0,1[$  of a real number $t$,
 such that $t-\{t\}\in \Z$ (see Figure \ref{figure1}).
\begin{example}\label{ex:dimone}
Let $V=\R$,   $\lattice=\Z$, and consider the cones $\R_{\geq 0}$ and $\R_{\leq0}$.
Then
 \begin{align*}
   S_\Z^{{\R}}(s+\R_{\geq 0})(\xi) & = -\frac{\e^{s\xi}}{\xi}, &
   S_\Z^{{\R}}(s+\R_{\leq0})(\xi) & = \frac{\e^{s\xi}}{\xi} \\
   \intertext{while}
   S_\Z^{\{0\}}(s+\R_{\geq 0})(\xi) &=\e^{s\xi}\frac{\e^{\{-s\}\xi}}{1-\e^{\xi}}, &
   S_\Z^{\{0\}}(s+\R_{\leq0})(\xi) &=\e^{s\xi}\frac{\e^{-\{s\}\xi}}{1-\e^{-\xi}}.
 \end{align*}
\end{example}

For $L\neq \{0\}$, the fact that
$S^L_\lattice(\p) $ is actually  an element of $\CM_{\ell}(V^*)$, is proven in \cite{so-called-paper-2}, as a
consequence of  explicit computations which we recall in the next section.

The valuation $\p\mapsto S^L_\lattice(\p)(\xi)$ extends by linearity to the space of linear combinations of
indicator functions of semi-rational polyhedra.
In particular, if $\p$ is a semi-open polytope defined by rational equations and inequalities,
$S^L_\lattice(\p)(\xi)$ is holomorphic and still given by Equation \eqref{eq:SL}.

Let us note an obvious but important property.
\begin{lemma}\label{lemma:lattice-shift}
  If $v\in \lattice+L $, let $v+\p$ be the shifted polyhedron, then
  $$
  S^L_\lattice(v+\p)(\xi)= \e^{\langle \xi,v\rangle} S^L_\lattice(\p)(\xi).
  $$
\end{lemma}

\subsection{Case of  a simplicial cone}\label{sect:simplicial}
We recall some results of \cite{so-called-paper-2}.

\medskip

We look first at the discrete generating function. Let $\c$ be a simplicial
cone and let  $v_i$, $i=1,\dots, d$,
be lattice generators of its edges (we do not assume that the $v_i$'s are primitive).
Let $\b=\sum_{i=1}^d [0,1\mathclose[\, v_i$, the corresponding \emph{semi-open cell}.
Let $\vol_\lattice(\b)= \mathopen| \det_\lattice(v_1,\dots,v_d) \mathclose|$ be its volume
  with respect to the Lebesgue measure defined by the   lattice.
  Then
\begin{equation}\label{eq:I}
 I_\lattice(s+\c)(\xi)=\e^{\la \xi,s \ra} \frac{(-1)^d \vol_\lattice(\b)} {\prod_{i=1}^d\la \xi,v_i\ra}.
 \end{equation}
The discrete generating function for the cone  can be expressed in terms of that of the semi-open cell:
\begin{equation}\label{eq:Sbox}
S_\lattice(s+\c)(\xi)=S_\lattice(s+\b)(\xi) \frac{1}{\prod_{i=1}^d(1- \e^{\la \xi,v_i\ra})}.
\end{equation}

It follows from \eqref{eq:Sbox} that the intermediate functions $S_\lattice(s+\c)(\xi)$
 belong to the space $\CM_{\ell}(V^*)$. More precisely, their poles are given by the edges of the cone.
\begin{lemma}\label{lemma:prodvj-S}
 Let $\c$ be a  polyhedral cone with edge generators
 $v_i,i=1,\dots, N$  and let   $s\in V$.
The function
$\prod_{i=1}^N\la \xi, v_i\ra S_\lattice(s+\c)(\xi)$ is holomorphic near $\xi=0$.
\end{lemma}
\begin{proof}
The case where  the cone $\c$ is simplicial follows immediately from \eqref{eq:Sbox},
as $S_\lattice(s+\b)(\xi)$ is holomorphic.
The general case follows from the valuation property,
by using a decomposition of $\c$ into simplicial cones without added edges.%
\footnote{Suppose $\c$ is generated by $v_i$, $i=1,\dots,N$. 
  Triangulating $\c$ without adding edges gives a
  primal decomposition of the form $[\c] = \sum_j [\c_j] + \sum_j \epsilon_j
  [\mathfrak l_j]$, where $\c_j$ are
  (full-dimensional) simplicial cones and $\mathfrak l_j$ are lower-dimensional cones
  that arise in an inclusion-exclusion formula with coefficients
  $\epsilon_j$. 
  Both $\c_j$ and $\mathfrak l_j$ are generated by subsets of $v_i$, $i=1,\dots,N$.
  As shown in \cite{koeppe-verdoolaege:parametric,Brion1997residue}, we can
  also construct decompositions that only involve full-dimensional cones. 
  For example, as shown in \cite{koeppe-verdoolaege:parametric}, we can find a
  decomposition of the form $[\c] = \sum_j [\tilde\c_j]$, where
  $\tilde\c_j$ are full-dimensional semi-open cones whose closures are $\c_j$.  
  Then, as shown in \cite{Brion1997residue},
  modulo indicator functions of cones with lines, we can replace the semi-open
  cones~$\tilde\c_j$ by closed cones $\bar\c_j$ and thus obtain a
  decomposition of the form 
  $[\c] \equiv \sum_j \epsilon_j [\bar\c_j]$ (modulo indicator functions of cones with
  lines), where $\epsilon_j\in\{\pm1\}$ and $\bar\c_j$ are simplicial cones
  which are generated by subsets of $\pm v_i$, $i=1,\dots, N$. 
}
\end{proof}
\medskip

Next, we consider the intermediate generating function in the case where   $\c$ is simplicial and $L$ is one of its faces. In this case,
the intermediate generating function $S^L_\lattice(s+\c)$
decomposes as a product.
For $I\subseteq \{1,\ldots, d\}$, we denote by $L_I$ the
linear span of the vectors $v_i$, $i\in I$.  Let $I^c$ be the complement of
$I$ in $\{1,\ldots ,d\}$.
For $x\in
V$, we  write $x=x_I + x_{I^c}$  with respect to the decomposition  $V=L_I\oplus L_{I^c}$.
Thus we identify the quotient $V/L_{I}$ with $L_{I^c}$ and   we
denote the projected lattice by $\lattice_{I^c}\subset L_{I^c}$.
Write   $\c_{I}$ for  the cone  generated by the vectors  $v_j$, for $j\in
I$ and $\c_{I^c}$ for the cone generated by  the vectors  $v_j$, for $j\in
I^c.$
The projection of the cone $\c$ on $V/L_I=L_{I^c}$ identifies with $\c_{I^c}$.
We write also  $\xi=\xi_I + \xi_{I^c}$,  with respect to the decomposition $V^*= L_I^*\oplus L_{I^c}^*$.
Then we have the product formula
\begin{equation}\label{eq:SLI}
 S^{L_I}_\lattice(s+\c)(\xi) =
 S_{\lattice_{I^c}}(s_{I^c}+\c_{I^c})( \xi_{I^c})\, I_{\lattice\cap L_I
 }(s_I+\c_I)(\xi_I).
\end{equation}
\smallbreak

Finally, the general case is reduced to the case where $\c$ is simplicial and
$L$ is parallel to one of its faces by the Brion--Vergne decomposition (Theorem 19 
%% it's 19, not 18 in the published version at least.
in \cite{so-called-paper-2}),
which we summarize in the following proposition and illustrate on an example in Figure~\ref{fig:brion-vergne-2d}.
\begin{figure}[t]
  %\centering
  \noindent\hspace*{-1em}
  \begin{tabular}{@{}l@{}l@{}}
 \begin{minipage}[c]{.35\linewidth}%
   \scalebox{.75}{\ifpdf
  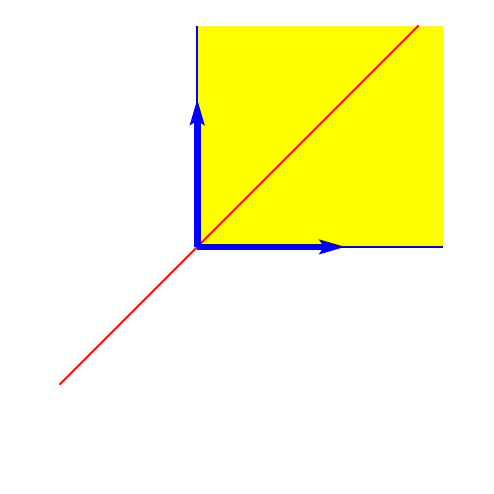
   \else
   \input{brionvergne-2d.pstex_t}
   \fi}%
 \end{minipage}%
 &\hspace*{-1em}{\Huge$\equiv$}\hspace*{-1em}%
 \begin{minipage}[c]{.31\linewidth}%
   \scalebox{.75}{%
     \ifpdf
  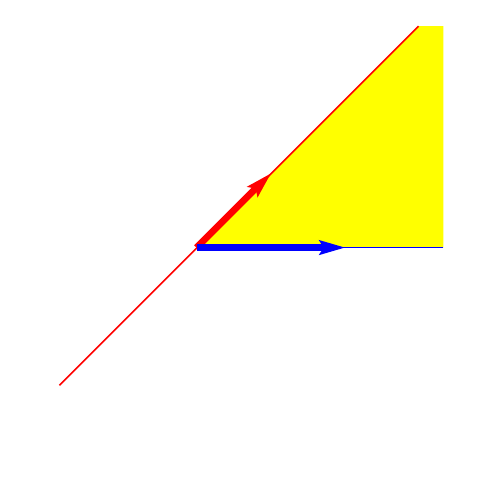
   \else
   \input{brionvergne-2d-cone1.pstex_t}
   \fi}%
 \end{minipage}%
 \hspace*{0em}{\Huge$-$}\hspace*{-1em}%
 \begin{minipage}[c]{.35\linewidth}%
   \scalebox{.75}{%
     \ifpdf
  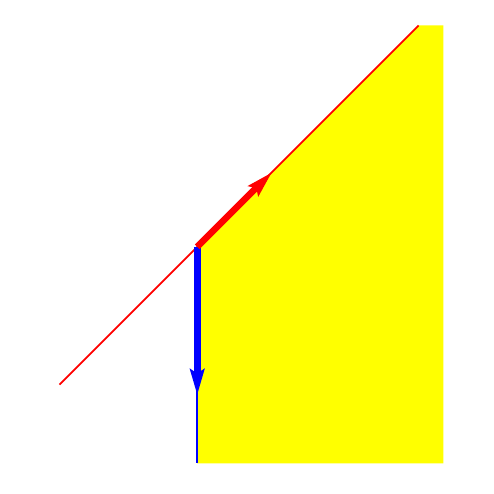
   \else
   \input{brionvergne-2d-cone2.pstex_t}
   \fi%
   }
 \end{minipage}\\[10ex]%
 &\hspace*{-1em}{\Huge$\equiv-$}\hspace*{0.5em}%
 \begin{minipage}[c]{.31\linewidth}%
   \scalebox{.75}{%
     \ifpdf
  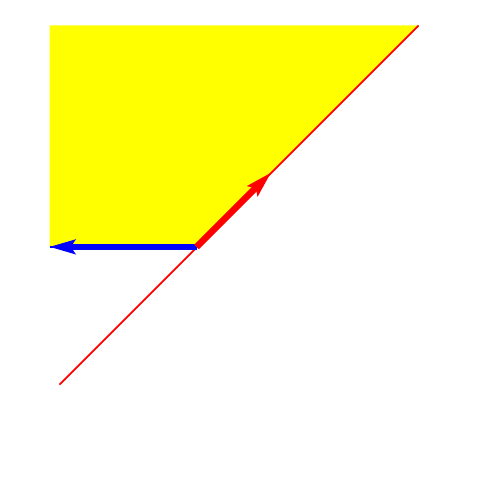
   \else
   \input{brionvergne-2d-cone1prime.pstex_t}
   \fi}%
 \end{minipage}%
 \hspace*{-1.5em}{\Huge$+$}\hspace*{0em}%
 \begin{minipage}[c]{.31\linewidth}%
   \scalebox{.75}{%
     \ifpdf
  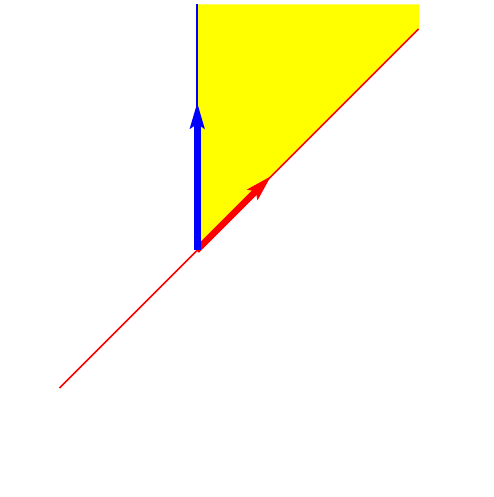
   \else
   \input{brionvergne-2d-cone2prime.pstex_t}
   \fi}%
 \end{minipage}%
\end{tabular}\vspace{-3ex}
 \caption{Two Brion--Vergne decompositions of a cone~$\c$ into cones with a face
   parallel to the subspace~$L$, modulo cones with lines.  % The vectors in the
   % quotient $V/L$ determine the signs~$\epsilon_{\sigma,j}$.
 }
 \label{fig:brion-vergne-2d}
\end{figure}

\begin{proposition}[Brion--Vergne decomposition]\label{prop:brion-vergne-decomposition}
  Let $\c$ be a full-dimen\-sional cone. Then there exists a decomposition of its indicator function
 $[\c]\equiv \sum_i\epsilon_i [\c_i]$, where each $\c_i$ is a simplicial full-dimensional cone
 with  a face parallel to $L$, and the congruence holds modulo the space spanned by  indicators functions of
 cones which contain  lines.
\end{proposition}
\begin{proof}
The case where $\c$ is simplicial is Theorem 19 in \cite{so-called-paper-2}. Moreover,
for any full-dimensional cone $\c$, there exists a decomposition  $[\c]\equiv
\sum_a\epsilon_a [\c_a]$ (modulo indicator functions of cones with lines), where $\c_a$ are
full-dimensional simplicial cones.\footnote{The most well-known way to construct such a
  decomposition is using the ``duality trick'' (going back to \cite{Brion88}):
  We triangulate the dual cone $\c^\circ \subseteq V^*$ and obtain a
  decomposition $[\c^\circ] \equiv \sum_a [\c^\circ_a]$ (modulo indicator
  functions of lower-dimensional cones of~$V^*$), where $\c^\circ_a$ are simplicial
  cones of~$V^*$.  This implies the decomposition $[\c^\circ] \equiv \sum_a
  [(\c^\circ_a)^\circ]$ (modulo cones with lines). 
}
\end{proof}

Lemma \ref{lemma:prodvj-S} holds also for intermediate generating functions $S^L_\lattice(s+\c)(\xi)$.
We will deduce it below from the Poisson summation formula for $S^L_\lattice(s+\c)(\xi)$, (Theorem \ref{th:poisson}),
and  the following weaker result.
\begin{lemma}\label{lemma:prodwj-SL}
For a given cone $\c$, there exist a finite number of vectors $w_j\in \lattice$ such that
$\prod_{j=1}^N\langle \xi,w_j\rangle S^L_\lattice(s+\c)(\xi)$ is holomorphic near $\xi=0$. In other words,  the function
$ S^L_\lattice(s+\c)(\xi)$ belongs to the space $\CM_{\ell}(V^*)$.
\end{lemma}
\begin{proof}
By Proposition \ref{prop:brion-vergne-decomposition} and the valuation property,
we have  $$ S^L_\lattice(s+\c)(\xi)= \sum_i\epsilon_i S^L_\lattice(s+\c_i)(\xi),$$ where each $\c_i$ is simplicial with a face parallel to $L$.
 However,  this process  may introduce new edges.
Finally, when $\c$ is a simplicial cone with a face parallel to $L$,
the result follows from   Formulas  \eqref{eq:SLI} and \eqref{eq:Sbox}.
\end{proof}

\begin{figure}\label{figure1}
\begin{center}
\includegraphics[width=5cm]{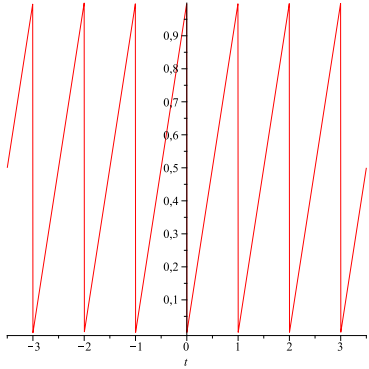} \\
\caption{The fractional part  $\{t\}$ }\label{b1}
%%% FIXME: Maybe redraw to show one-sided discontinuity
 \end{center}
\end{figure}
Let us give some examples when $V$ is two-dimensional.
\begin{example}[positive quadrant]\label{ex:standard-dim2}
Let $V=\R^2$, $\lattice=\Z^2$, and let $\c$ be the positive quadrant.
For $s=(s_1,s_2)$ and $\xi=(\xi_1,\xi_2)$,  $ S^L_\Z(s+\c )(\xi)$ is given by
\begin{equation}
\begin{aligned}
   &{\e^{s_1\xi_1+s_2\xi_2}}{\frac {{\e^{ \{ -s_1 \} \xi_1}}{\e^{ \{ -s_2 \} \xi_2}}}{ \left( 1-{\e^{\xi_1}} \right)
     \left( 1-{\e^{\xi_2}} \right) }},  &&\mbox {  if } L=\{0\}, \\
 & {\e^{s_1\xi_1+s_2\xi_2}} \left(\frac{\e^{ \{-( s_1+s_2) \} \xi_1}}{1-\e^{\xi_1}}-\frac{\e^{ \{ -(s_1+s_2) \} \xi_2}}{1-\e^{\xi_2}}\right)
\frac{1}{\xi_1-\xi_2},  &&\mbox {  if } L=\R(1,-1),\\
&{\e^{s_1\xi_1+s_2\xi_2}}
\left(\frac{\e^{\{s_2-s_1\}\xi_1}}{1-\e^{\xi_1}}-\frac{\e^{-\{s_2-s_1\}\xi_2}}{1-\e^{-\xi_2}}\right)\frac{-1}{\xi_1+\xi_2},
&&\mbox {  if } L=\R(1,1).
\end{aligned}
\label{eq:standard-dim2}
\end{equation}
The first formula just follows from \eqref{eq:Sbox}.
The second and the third formula follow from Brion--Vergne decompositions and
the product formula~\eqref{eq:SLI}.
The decomposition used for the third formula, for $L=\R(1,1)$, is
$$[\c]=[\c_1]-[\c_2]- [\{x_1\geq 0\}],$$ where $\c_1=\{x_2\geq 0,x_1-x_2\geq 0\}$ and
$\c_2= \{x_1 \geq 0,x_1-x_2\geq 0\}$, as depicted in
Figure~\ref{fig:brion-vergne-2d} (top).
The decomposition is not unique.    We can compute  $S^L(s+\c)$
using the other Brion--Vergne decomposition, depicted in Figure~\ref{fig:brion-vergne-2d} (bottom),
 $$[\c]=-[\c'_1]+[\c'_2]+ [\{x_2\geq 0\}],$$ where $\c'_1=\{x_2\geq 0,x_1-x_2\leq 0\}$ and
$\c'_2= \{x_1 \geq 0,x_1-x_2\leq 0\}$.
Then  we obtain
an expression of $S^L(s+\c)$ in terms of $\{s_1-s_2\}$ instead of $\{s_2-s_1\}$.
\end{example} 

\subsection{Intermediate sum $S^L_\lattice(s+\c)(\xi)$ as a function of  the pair $(s,\xi)$. Step-polynomials and quasi-polynomials on a rational space.}
In this section, $\c$ is a rational cone of full dimension $d$.

\subsubsection{The function $\retroS^L_\lattice(s,\c)(\xi)$}
 We  study  the properties of  $S^L_\lattice(s+\c)(\xi)$
considered as a function of the two variables   $s\in V$, $\xi\in V^*$.
Actually, as we will see, these properties are more striking and useful when read on the
function
\begin{equation}\label{eq:functionM}
\retroS^L_\lattice(s,\c)(\xi)= \e^{-\langle \xi,s\rangle}S^L_\lattice(s+\c)(\xi).
\end{equation}
The following lemma is immediate.
\begin{lemma}\label{lemma:s-shift}
Let $\p$ be a semi-rational polyhedron. Let
$s\in V$, and let
 $ \retroS^L_\lattice(s,\p)(\xi)=\e^{-\langle \xi,s\rangle} S^L_\lattice(s+\p )(\xi)$.
Then, for $v \in \lattice+L$, we have
$\retroS^L_\lattice(s+v,\p )(\xi)= \retroS^L_\lattice(s,\p )(\xi)$.
\end{lemma}
Thus  the function $s \mapsto \retroS^L_\lattice(s,\p)(\xi)$ can be considered as a function on $V/L$
which is periodic with respect to the projected lattice $\lattice_{V/L}$.
We will often drop the subscript $\lattice$.
\begin{example}[Continuation of Example \ref{ex:standard-dim2}] \label{ex:standard-dim2-retroSL} 
$\retroS^L_\Z(s,\c )(\xi)$ is given by:
\begin{subequations}
\begin{align}
  & {\frac {{\e^{ \{ -s_1 \} \xi_1}}{\e^{ \{ -s_2 \} \xi_2}}}{ \left( 1-{\e^{\xi_1}} \right)
      \left( 1-{\e^{\xi_2}} \right) }},  && \mbox {  if } L=\{0\}, \\
  & \left(\frac{\e^{ \{-( s_1+s_2) \} \xi_1}}{1-\e^{\xi_1}}-\frac{\e^{ \{ -(s_1+s_2) \} \xi_2}}{1-\e^{\xi_2}}\right)
  \frac{1}{\xi_1-\xi_2},  && \mbox {  if } L=\R(1,-1), \\
  & \left(\frac{\e^{\{s_2-s_1\}\xi_1}}{1-\e^{\xi_1}}-\frac{\e^{-\{s_2-s_1\}\xi_2}}{1-\e^{-\xi_2}}\right)\frac{-1}{\xi_1+\xi_2},
  && \mbox {  if } L=\R(1,1).\label{eq:standard-dim2-ML}
  \end{align}
\end{subequations}
\end{example}

\subsubsection{Step-polynomials and quasi-polynomials on $V$}

A crucial role in our study is played by the individual homogeneous components
$S^L(s+\c )_{[m]}(\xi)$ and $ \retroS^L(s,\c )_{[m]}(\xi)$. A pleasant feature
is  that, when $\c$ is fixed, the homogeneous component of $\xi$-degree $m$   can be viewed as a
function of $s\in V$ with values in a finite-dimensional vector space, namely
the space  of rational functions of homogeneous $\xi$-degree $m$ which can be
written in  the form $\frac{P(\xi)}{\prod_j \langle \xi,w_j\rangle}$, where
the family of vectors $w_j$ is given by Lemma \ref{lemma:prodwj-SL}. 
\tgreen{Changed from `total degree' to `homogeneous degree' in above, to match 3polys.}

We introduce  an algebra of functions on $V\times V^*$ in order to
describe these homogeneous components.
Let us start with an example.
\begin{example}[Continuation of Example \ref{ex:standard-dim2}]\label{ex:standard-dim2-homogeneous}
Let  $L=\R(1,1)$. We expand \eqref{eq:standard-dim2-ML}, using the Bernoulli polynomials, defined by
\begin{equation}\label{eq:bernoulli-generating-function}
  \e^{tz}\frac{z}{\e^z-1}=\sum_{n=0}^\infty \frac{B_n(t)}{n!}z^n,
\end{equation}
obtaining
\allowdisplaybreaks
\begin{align*}
  \retroS^L(s,\c )_{[-2]}(\xi)&=\frac {1}{\xi_1\xi_2}, \\
   \retroS^L(s,\c )_{[-1]}(\xi) &= 0,\\
   \retroS^L(s,\c )_{[0]}(\xi)  &= \frac{B_2(\{s_2-s_1\})}{2}=
   \frac{\{s_2-s_1\}^2-\{s_2-s_1\}+\frac{1}{6}}{2},\\
    \retroS^L(s,\c )_{[1]}(\xi)& = \frac{B_3(\{s_2-s_1\})}{3!}(\xi_1-\xi_2)\\
     &= \frac{
       \{s_2-s_1\}^3-\frac{3}{2}\{s_2-s_1\}^2+\frac{1}{2}\{s_2-s_1\}}{6}(\xi_1-\xi_2).\\
     \intertext{From these formulas, we obtain}
S^L(s+\c )_{[-2]} (\xi)&= \frac {1}{\xi_1\xi_2}, \\
 S^L(s+\c )_{[-1]} (\xi)&= \frac {s_1 \xi_1 +s_2\xi_2}{\xi_1\xi_2}, \\
 S^L(s+\c )_{[0]}  (\xi)&=  \frac{(s_1 \xi_1 +s_2\xi_2)^2
  +\xi_1 \xi_2 ( \{s_2-s_1\}^2-\{s_2-s_1\}+\frac{1}{6})
   }{2\xi_1\xi_2}.
\end{align*}
\end{example}
We observe that  the numerators above
are written as   polynomial functions of $s_1,s_2, \{s_2-s_1\},\xi_1,\xi_2$.
 We next describe the general case.

Let $V^*_\Q=\lattice^*\otimes \Q$ be the set of rational elements of $V^*$.
 \begin{definition}\label{def:step-poly-V}
 $\polypp(V)$ is the algebra of functions on $V$ generated by the functions
 $s\mapsto \{\langle\lambda,s\rangle\}$, where $\lambda\in V^*_\Q$.
 An element of $\polypp(V)$  is called a (rational) \emph{step-polynomial} on $V$.
 %% I added `rational' for the benefit of a discussion of the more general
 %% case in 3Polys. --Matthias
 \end{definition}
 \begin{remark}
   Note that there are many relations between these generators. For example,
   if $V = \R$, consider for $\lambda\in\Q$ the function $f_\lambda(s) = 1 - (\{\lambda s\} +
   \{-\lambda s\})$, which is $1$ if $\lambda s$ is an integer and $0$ otherwise. Then we have
   the polynomial relation $f_1(s) = f_2(s) f_3(s)$.
 \end{remark}

 Since the generators $s\mapsto \{\langle\lambda,s\rangle\}$ are bounded
 functions, it follows that a step-polynomial is a bounded function on $V$.

  The space $\polypp(V)$ has a natural filtration,
 where    $\polypp_{[\leq k]}(V)$ is  the subspace generated by products of at most $k$ functions
 $\{\langle \lambda,s\rangle\}$.

 For $\eta\in\lattice^*$,
the function $s\mapsto \{\langle \eta,s\rangle\}$ is  $\lattice$-periodic.
For a given step-polynomial $f(s)$, let $q\in \N$
 be such that $q\lambda\in \lattice^*$ for all the $\lambda$'s involved in an expression of
 $f(s)$ (such an expression is not unique).  Then  $f$  is $q\lattice$-periodic.
 
 Next, we consider the algebra of functions on $V$  generated by $\polypp(V)$ and  $\CP(V)$,
where $\CP(V)$ is the algebra of polynomial functions on $V$.
It is clear that this algebra is the tensor product
$\polypp(V) \otimes\CP(V) $. We denote it by $\polypp\CP(V)$.
\begin{definition}\label{def:quasi-polynomial-V}
 The elements of $\polypp\CP(V)=\polypp(V) \otimes\CP(V)$ are called \emph{quasi-polynomials} on $V$.
\end{definition}
Now, let $\Psi$ be a finite subset of $V_\Q^*$.  %% Changed to V_\Q^* from \Lambda^* to make uniform with 3Polys. --Matthias
There corresponds a subalgebra of quasi-polynomials on $V$.
\begin{definition}\label{def:polypp}
  \begin{enumerate}[\rm(i)]
  \item $\polypp^\Psi(V)$ is the algebra of functions on $V$ generated by the
    functions $s\mapsto\{\langle \eta,s\rangle\}$, with $\eta\in \Psi$.
  \item $\polypp^\Psi_{[\leq k]}(V)$ is  the subspace of $\polypp^\Psi(V)$ generated by products of at most $k$ functions
    $\{\langle \eta,s\rangle\}$, with $\eta\in\Psi$.
  \item 
    $\polypp\CP^{\Psi}(V)$ is the algebra of functions on $V$ generated by
    $\polypp^{\Psi}(V)$ and $\CP(V)$.
  \end{enumerate}
\end{definition}
\begin{figure}
\begin{center}
\includegraphics[width=8cm]{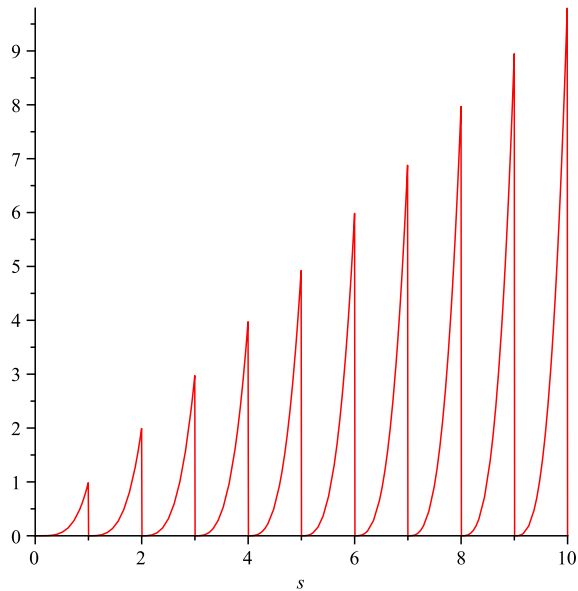} \\
 \caption{The function $s\mapsto \{s\}^3 s$}\label{figure2}
 \end{center}
\end{figure}
The quasi-polynomials in $\polypp\CP^\Psi(V)$ are piecewise polynomial, in a sense which we describe now.
\begin{definition}\label{def:Psi-alcove}
 Let $\Psi$ be a finite subset of $V_\Q^*$.  
 %% Changed to V_\Q^* from \Lambda^* to make uniform with 3Polys. --Matthias
 We consider  the  hyperplanes in $V$ defined by the equations
 $$
 \langle \eta ,x\rangle= n \quad \mbox{  for  } \eta\in \Psi  \mbox{ and   }   n\in \Z.
 $$
 A connected component of the complement of the union of these hyperplanes
 in~$V$ is called a \emph{$\Psi$-alcove}.
 \end{definition}
Thus, an alcove is the interior of a polyhedron whose faces are some of  the hyperplanes above.
If $\Psi$  generates $V^*$, all alcoves are bounded. Otherwise, they are  unbounded.
If $s$ is any element of $V$, and $v\in V$ is a fixed element such that
$\langle\eta,v\rangle \neq 0 $ for $\eta\in \Psi$,
then the  curve $s+tv$  is contained  in a $\Psi$-alcove for small $t<0$.

If $\eta\in \Psi$, the restriction to any $\Psi$-alcove
of the function $s\mapsto \{\langle \eta ,s\rangle\}$ is affine. Therefore the restriction  of
a quasi-polynomial $f(s)\in \polypp_{[\leq k]}^{\Psi}(V) \otimes\CP_{[r]}(V)$ to an alcove
is a polynomial in~$s$ of degree $\leq k+r$.  This motivates the following
definition.
\begin{definition}\label{def:quasipoly-filtration}
  \begin{enumerate}[\rm(i)]
  \item A function $f(s)\in \polypp_{[\leq k]}^{\Psi}(V) \otimes\CP_{[r]}(V)$
    is said to be of \emph{polynomial
      degree} $r$ and of \emph{local degree} (at most) $k+r$.
  \item 
    We define $\polypp\CP^{\Psi}_{[\leq q]}(V)$  to be the subspace of quasi-polynomials of
    local degree at most  $q$.
  \end{enumerate}
\end{definition}
In Figure~\ref{figure2}, we draw the graph of a quasi-polynomial function on
$\R$ (with $\Lambda=\Z$ and  $\Psi=\{1\}$) of local degree~$4$.
\tgreen{Changed above from `bidegree' (second meaning) to `local degree'.}

\subsubsection{Properties of homogeneous components of generating functions of shifted cones}
We start with the case $L=\{0\}$. Given a cone $\c\subset V$,
we define a subalgebra of step-polynomials associated
with $\c$.
The fundamental fact here is the existence of a decomposition of the indicator function of $\c$ as a signed sum
\begin{equation}\label{eq:decomp-unimodular}
 [\c]\equiv \sum_{\u} \epsilon_{\u} [\u],
\end{equation}
where the cones $\u$ are unimodular, and the congruence is modulo
the space spanned by indicator functions of cones which contain  lines. If $\c$ is full-dimensional, we can assume that the cones $\u$ are also full-dimensional,  see \cite{barvinokzurichbook} for instance.
By the valuation property, we have
$$
S(s+\c )(\xi)=\sum_{\u} \epsilon_{\u} S(s+\u )(\xi).
$$
For each unimodular cone $\u$ in \eqref{eq:decomp-unimodular},
let $v_j^\u\in\Lambda$, $1\leq j\leq d $, be the primitive generators of the cone $\u$, and let
 $\eta_j^{\u}\in\Lambda^*$, $1\leq j\leq d $, be the dual basis.
 \begin{definition}\label{def:Psi}
We denote by  $\Psi_\c \subset \lattice^*$  the set of all $\eta_j^{\u}$, for $j=1,\ldots,d$,
 where $\u$ runs over the set  of unimodular cones  entering in the  decomposition~\eqref{eq:decomp-unimodular} of $[\c]$.
\end{definition}
$\Psi_\c$ depends of the choice of  the decomposition, but we do not record it in the notation, for brevity.

We can now state the important bidegree properties of the homogeneous components of the functions $S(s+\c )(\xi)$
 and $ \retroS(s,\c)(\xi)$.
Here ``bidegree'' refers to the interaction between the (local) degree in~$s$ and the
homogeneous degree in~$\xi$.  % In the case of $S(s+\c )(\xi)$, the relevant
% degree in~$s$ is the local degree; in the case of $\retroS(s,\c)(\xi)$, it is
% the degree as step-polynomials. 
\tgreen{Added the two previous sentences to explain what `bidegree' is
  supposed to mean.  Updated 2018-08-31 regarding Michele's comment.}
\begin{theorem}\label{prop:homogeneous-M}
Let $m\in \Z$.
\begin{enumerate}[\rm (i)]
\item  The function $(s,\xi)\mapsto \retroS(s,\c)_{[m]}(\xi)$
  belongs to the space $$\polypp_{[\leq m+d]}^{\Psi_\c}(V)\otimes
  \CR_{[m]}(V^*).$$
\item  The function $(s,\xi)\mapsto S(s+\c)_{[m]}(\xi)$ belongs to
  the space $$\polypp\CP^{\Psi_\c}_{[\leq m+d]}(V)\otimes \CR_{[m]}(V^*).$$
  More precisely,
$$
S(s+\c)_{[m]}(\xi)=\sum_{r=0}^{m+d} \frac{\langle\xi,s\rangle^r}{r!}
\retroS(s,\c)_{[m-r]}(\xi).
$$
\item The homogeneous component in $\xi$ of lowest degree has degree $m = -d$ and does not depend on $s$. It is given by the integral
$$
S(s+\c)_{[-d]}(\xi)=\retroS(s,\c)_{[-d]}(\xi)= I(\c)(\xi).
$$
\end{enumerate}
\end{theorem}

To rephrase (ii), we can say that the numerator of the homogeneous component $S(s+\c)_{[m]}(\xi)$ is a quasi-polynomial
function of $s$ (with coefficients polynomials in $\xi$), and the local degree
in~$s$ of this quasi-polynomial (and so its complexity)  grows with the
homogeneity degree~$m$ in $\xi$. 
\tgreen{Changed from bidegree to local degree in previous sentence.}

\begin{proof}
Let $\u$ be one of the unimodular cones in the decomposition of~$[\c]$.
 We write $\xi=\sum_j \xi_j \eta_j^\u$. Then
$S(s+\u)(\xi)$ is directly computed by summing a multiple geometric series (cf.\ Example \ref{ex:dimone}),
hence
\begin{align*}
\retroS(s,\u)(\xi)&=\exp \Bigl({\sum_j \{-\langle
    \eta_j^u,s\rangle\}\xi_j}\Bigr) \prod_{j=1}^d
\frac{1}{1-\e^{\xi_j}},\\
\retroS(s,\u)_{[m]}(\xi)&= \sum_{k=0}^{d+m}\frac{(\sum_j \{-\langle
  \eta_j^u,s\rangle\}\xi_j )^k}{k!}  \biggl(\prod_{j=1}^d
\frac{1}{1-\e^{\xi_j}}\biggr)_{[m-k]}.
\end{align*}
In this formula, it is clear that the $k$-th term  belongs to
$\polypp_{[\leq k]}^{\Psi_\c}(V)\otimes \CR_{[m]}(V^*)$. As $k\leq m+d$, we obtain (i).
The homogeneous components $S(s+\c)_{[m]}(\xi)$
are immediately computed out of those of $\retroS(s,\c)(\xi)$, hence (ii).
Part (iii) was proved in \cite{so-called-paper-1}, Lemma~16.
\end{proof}

Now let $L\subseteq V$ be any rational subspace.
In order to obtain a similar result for the intermediate generating function $S^L(s+\c)(\xi)$,
we follow the steps of the proof of Lemma \ref{lemma:prodwj-SL}.
First, we decompose $[\c]$ into a signed sum of simplicial cones with a face parallel to $L$.
For each of these cones, we decompose the projected cone in $V/L$
 into a signed sum of  cones which are unimodular with respect to the projected lattice
$\lattice_{V/L}$. We thus have a collection of unimodular cones $\u\subset V/L$.
When $(V/L)^*$ is identified with $L^\perp \subset V^*$,
the dual lattice
$(\lattice_{V/L})^*$ is identified with $\lattice^*\cap L^\perp$.
For each  $\u$, we let $v_j^\u\in\Lambda_{V/L}$ be primitive edge generators of $\u$ and 
consider the dual basis $\eta_j^\u\in \lattice^*\cap L^\perp$.
\begin{definition}\label{def:PsiL}
We denote by $\Psi_\c^L \subset\lattice^*\cap L^\perp$ the set  of all $\eta_j^\u$.
\end{definition}
Then the functions in $\polypp^{\Psi_\c^L}(V)$ are   functions on $V/L$ and  are  $\lattice_{V/L}$-periodic.
 Using the product formula \eqref{eq:SLI},
    the proof of the following theorem  is  similar to the case $L=\{0\}$ (Theorem~\ref{prop:homogeneous-M}).
\begin{theorem}\label{prop:homogeneous-ML}
Let $m\in \Z$.
\begin{enumerate}[\rm (i)]
\item The function $(s,\xi)\mapsto \retroS^L(s,\c)_{[m]}(\xi)$
  belongs to the space $$\polypp_{[\leq m+d]}^{\Psi_\c^L}(V)\otimes
  \CR_{[m]}(V^*).$$

\item The function $(s,\xi)\mapsto S^L(s+\c )_{[m]}(\xi)$
  belongs to the space
  $$\polypp\CP_{[\leq m+d]}^{\Psi_\c^L}(V)\otimes \CR_{[m]}(V^*).$$ 
  More
  precisely
  \begin{equation}\label{eq:homogeneous-ML}
    S^L(s+\c )_{[m]}(\xi)=\sum_{r=0}^{m+d} \frac{\langle\xi,s\rangle^r}{r!}  \retroS^L(s,\c )_{[m-r]}(\xi).
  \end{equation}
\item The homogeneous component in $\xi$ of lowest degree has degree $m=-d$ and does not depend on
  $s$. It is given by the integral 
  $$
  S^L(s+\c)_{[-d]}(\xi)=\retroS^L(s,\c)_{[-d]}(\xi)= I(\c)(\xi).
  $$
\end{enumerate}
\end{theorem}
\begin{remark}
    If $L=V$, then $\Psi_\c^L$ is empty, hence  $\polypp^{\Psi_\c^L}(V)$ is just the scalars. Indeed,
    $\e^{-\langle\xi,s\rangle}I(s+\c)(\xi)= I(\c)(\xi)$ does not depend on $s$.
\end{remark}
\begin{remark}
  As we showed in \cite[Theorems 31 % computation of SL generating function
  and 38% extraction of Ehrhart coefficients
  ]{so-called-paper-1} and \cite[Theorems
  24 % computation of SL generating function
  and 28% extraction of Ehrhart coefficients
  ]{so-called-paper-2}, 
  the computation of these functions and their homogeneous components 
  can be made effective, and the bidegree structure, i.e., the interaction of
  the local degree in~$s$ and the 
  homogeneous degree in~$\xi$, takes a key role in 
  extracting the refined asymptotics.  
  \tgreen{Rephrased previous sentence to explain bidegree again; corrected again
    according to Michele's 2014-08-31 comment.}
  We have developed Maple implementation of such
  algorithms, which work with a symbolic vertex~$s$; the resulting formulas
  are naturally valid for any real vector $s\in V$.
\end{remark}

\subsection{One-sided continuity}

The meromorphic functions $M^L(s,\c )(\xi)$ and $S^L(s+\c )(\xi)$ and their homogeneous
components $M^L(s,\c )_{[m]}(\xi)$ and $S^L(s+\c )_{[m]}(\xi)$ 
enjoy some continuity  properties  when $s$ tends to $s_0$ along some directions that we will describe.
Let us look again at the simplest example  (Example~\ref{ex:dimone}). We have
 $$
S_\Z^{\{0\}}(s+\R_{\geq 0})(\xi) =\e^{s\xi}\frac{\e^{\{-s\}\xi}}{1-\e^{\xi}},
$$
and
$$
S_\Z^{\{0\}}(s+\R_{\leq0})(\xi) =\e^{s\xi}\frac{\e^{-\{s\}\xi}}{1-\e^{-\xi}}.
$$
Then observe that as functions of $s$, the first formula is continuous from
the left, while the second formula is continuous from the right. These
directions are the opposite  of the direction of the corresponding cones, and
the result is intuitively clear. For instance,  the set $(s+\R_{\geq 0})\cap \Z$
itself does not change when $s$ is moved slightly to the left; but it does
change if $s$ is moved slightly to the right from an integer.

We state a generalization of this result in higher dimensions.
In order to state the result, we observe that, in Theorem~\ref{prop:homogeneous-ML},
the infinite-dimensional space $\CR_{[m]}(V^*)$
can be replaced by the following finite-dimensional subspace.
\begin{definition}
  Let $(v_j)_{j=1}^N$ be a set of vectors in $V$ such that 
  the function
  $\prod_{j=1}^N
  \langle \xi,v_j\rangle S^L(s+\c)(\xi)$ is holomorphic near $\xi=0$. Then as
  rational functions of $\xi$, both $\retroS^L(s,\c )_{[m]}(\xi)$ and
  $S^L(s+\c )_{[m]}(\xi)$ lie in the finite-dimensional space of functions
  $f(\xi)$ such that $\prod_{j=1}^N \langle \xi,v_j\rangle f(\xi)$ is a
  polynomial of degree $m+N$.  We denote this space by $\frac{1}{\prod_{j=1}^N
    v_j}\CP_{[m+N]}(V^*)$.
\end{definition}
\begin{figure}
  \centering
  %% Variant with projected cone:\\
  %% \inputfig{continuity-3}\\
  %% Variant with cone $L-\c$:\\
  \ifpdf
  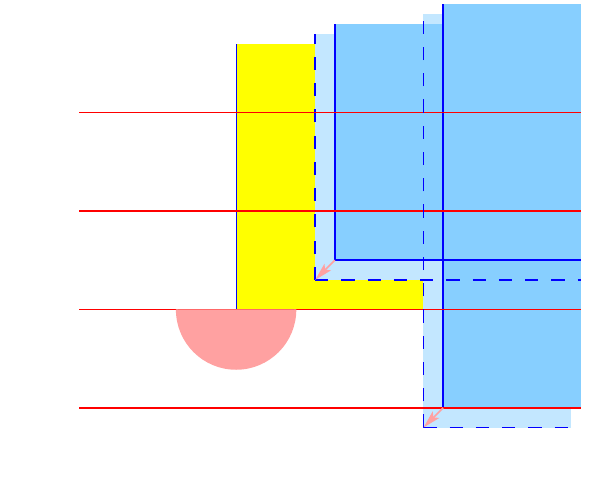
   \else
   \input{continuity-3-variant.pstex_t}
   \fi
  \caption{One-sided continuity of the functions $M^L(s,\c )(\xi)$ and $S^L(s+\c )(\xi)$ and their homogeneous
    components, as functions of the apex~$s$.  
    Discontinuities arise when one of the copies
    of~$L$ in $L+\lattice$ intersects the boundary of the cone~$s+\c$; in this
    case we still have one-sided continuity when we move in
    a direction $v \in L-\c$.
  }
  \label{fig:continuity-3}
\end{figure}
\begin{figure}
  \centering
  % Variant with projected cone:\\
  % \inputfig{continuity-1}\\
  % Variant with cone $L-\c$:\\
  \ifpdf
  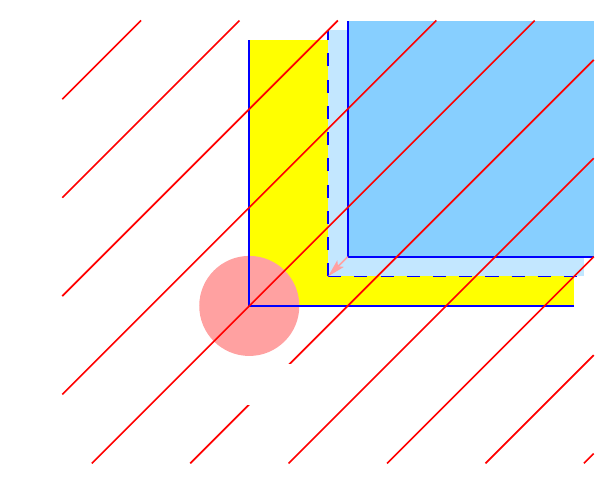
   \else
   \input{continuity-1-variant.pstex_t}
   \fi
  \caption{One-sided continuity of the functions $M^L(s,\c )(\xi)$ and $S^L(s+\c )(\xi)$ and their homogeneous
    components, as functions of the apex~$s$.  In this example, $L-\c = V$,
    and thus the functions actually depend
    continuously on~$s$.}
  \label{fig:continuity}
\end{figure}
\begin{proposition}\label{prop:alcoves}
  Let $(v_j)_{j=1}^N$ be a set of vectors in $V$ such that 
  the function
  $\prod_{j=1}^N
  \langle \xi,v_j\rangle S^L(s+\c)(\xi)$ is holomorphic near $\xi=0$. 
  \begin{enumerate}[\rm (i)]
  \item The restriction of $s\mapsto \retroS^L(s, \c)_{[m]}(\xi)$ and
    $s\mapsto S^L(s+\c)_{[m]}(\xi)$ to a $\Psi_\c^L$-alcove $\a\subset V$ are
    polynomial functions of $s$ with values in the finite-dimensional space
    $\frac{1}{\prod_{j=1}^Nv_j}\CP_{[m+N]}(V^*)$.
  \item Let $s\in V$. For any $v\in L-\c$, we have the following
    one-sided limit (cf.~Figures
    \ref{fig:continuity-3}~and~\ref{fig:continuity}),
 $$
 \lim_{\substack{t\to 0\\ t>0}}\retroS^L(s+tv, \c )_{[m]}(\xi)= \retroS^L(s, \c
 )_{[m]}(\xi).
  $$
  %%% I rephrased this in terms of the cone $L-\c$, which
  %%% simplifies the illustrations.  Thus $t$ is positive now!  --Matthias
\end{enumerate}
\end{proposition}
\begin{proof}
  Part (i) is an immediate consequence of Theorem~\ref{prop:homogeneous-ML}.

For part (ii) we assume that $\c$ is pointed, otherwise there is nothing to prove.
Recall the definition of $S^L(s+\c )(\xi)$.
 Fix $s_0\in V$. There is a non-empty open subset $U\subset V^*_\C$ such that for
$s $ near $s_0$ and $\xi \in U$, the following sum
$$
S^L(s+\c )(\xi)= \sum_{y\in \lattice_{V/L}}
\int_{(s+\c) \cap (y+L)} \e^{\la\xi,x\ra}\,\mathrm dx
$$
converges uniformly to a holomorphic function of $\xi$.
Intuitively, the lemma is based on the  observation that if  $v\in L-\c$ and $t>0$ is small enough,   then
the projections of the shifted cones $s+\c$ and $s+tv+\c$ on $V/L$ have the same lattice points.

Let us look first at the  extreme cases, $L=\{0\}$ or $L=V$.
If $L=\{0\}$, for $v\in L-\c$ and $t>0$ small enough, the shifted cones $s+\c$ and $s+tv+\c$ have the same lattice points, hence
 $S(s+tv+\c )(\xi)= S(s+\c )(\xi)$ for $\xi \in U$. It follows that these meromorphic functions are equal.
 If $L=V$, then $S^L(s+\c)(\xi)=I(s+\c)(\xi)=\e^{\langle \xi,s\rangle }
 I(\c)(\xi)$ depends continuously on the apex~$s$.

Now, let $L$ be arbitrary.
Let  $v\in L-\c$ and $t>0$ small enough.  Then
the projections of the shifted cones $s+\c$ and $s+tv+\c$ on $V/L$ have the same lattice points.
Consider a given $y\in \lattice_{V/L}$. If $y$ does not lie in this projection, then
$\int_{(s+tv +\c) \cap (y+L)} \e^{\la\xi,x\ra}\,\mathrm dx= 0$. Otherwise, it is clear that
the integral  $\int_{(s+tv +\c) \cap (y+L)} \e^{\la\xi,x\ra}\,\mathrm dx$ depends continuously on~$t$. Hence,
 for all $y\in \lattice_{V/L}$, we have
$$
 \lim_{\substack{t\to 0\\ t>0}}\; \int_{(s+tv +\c) \cap (y+L)} \e^{\la\xi,x\ra}\,\mathrm dx=
\int_{(s+\c) \cap (y+L)} \e^{\la\xi,x\ra}\,\mathrm dx,
$$
uniformly for $\xi\in U$.
Therefore $$\lim_{\substack{t\to 0\\ t>0}}S^L(s+tv+ \c )(\xi)= S^L(s+ \c )(\xi),$$
uniformly for $\xi\in U$.
The difficulty is that  $0\notin U$.
To deal with it, it is enough to prove that  there exists a
finite set of vectors $v_j\in V$ and a ball $B\subset  V^*_\C$ of center $0$ intersecting $U$, such that
$\prod_{j=1}^N \langle\xi,v_j\rangle S^L(s+ \c )(\xi)$ is holomorphic
and  uniformly bounded on $B$, for $s$ in neighborhood of a given $s_0$.
By the Montel compactness % compacity
theorem, it will follow that
$$
\lim_{\substack{t\to 0\\ t>0}}\;\prod_{j=1}^N \langle\xi,v_j\rangle S^L(s+tv+ \c )(\xi)= \prod_{j=1}^N \langle\xi,v_j\rangle  S^L(s+ \c )(\xi)
$$
uniformly for $\xi\in B$, therefore the limit will hold also for homogeneous components.

To prove the uniform boundedness property above,
we use the Brion--Vergne decomposition (Proposition~\ref{prop:brion-vergne-decomposition})
of $\c$ as a signed sum of simplicial cones,
each with a face parallel to~$L$, modulo cones with lines. We take  $(v_j)$
 to be the collection of all edge generators for all these cones.
So now we need only prove uniform boundedness for a simplicial cone for which $L$ is a face.
By the product formula, we are reduced to the extreme cases $S(s+\c)$ and $I(s+\c)$.
The latter is continuous with respect to $s$, so  uniform boundedness holds.
For the discrete sum,  uniform boundedness follows from Formula \eqref{eq:Sbox}.
\end{proof}
\begin{remark} Although the result is intuitively clear, a proof is needed because we have an infinite sum.  Indeed, consider the sequence of holomophic functions $f_n(z)=\frac{\e^{nz}-1}{ n z}$. This sequence  converges pointwise to~$0$ for  $\Re z<0$. However, the homogeneous components $\frac{n^k}{(k+1)!} z^k$ do not converge to $0$.
\end{remark}%
\begin{example}
  Let us look again at the positive quadrant from Examples
  \ref{ex:standard-dim2} and \ref{ex:standard-dim2-homogeneous},
  with $L=\R(1,1)$ (see Figure~\ref{fig:continuity}). By means of the Brion--Vergne decomposition of
  the quadrant depicted in 
  Figure~\ref{fig:brion-vergne-2d} (top), we computed a formula for $\retroS^L(s,\c)_{[m]}(\xi)$
  in terms of the Bernoulli  polynomial $B_{m+2}(\{s_2-s_1\})$. Let
  $v_1<0,v_2<0$, so $v\in -\c$.
Then  $\lim_{t\to 0, t>0}\{s_2+tv_2-s_1-tv_1\}=\{s_2-s_1\}$ if $v_2-v_1>0$. But if  $v_2-v_1<0$ and
$s_2-s_1\in \Z$, then $\lim_{t\to 0, t>0}\{s_2+tv_2-s_1-tv_1\}=1$ while $\{s_2-s_1\}= 0$.
However, observe that if $t>0$ is small, then $s_2+tv_2-s_1-tv_1\notin \Z$, hence
$$
\{s_2+tv_2-s_1-tv_1\}=1-\{-(s_2+tv_2-s_1-tv_1)\}.
$$
Thus  the limit statement in the lemma
 holds also for $v_2-v_1<0$,  due to  the property of Bernoulli polynomials:
$B_n(1-u)=(-1)^n B_n(u)$.

Another  way to reach this conclusion is to use the other % observe
% that we computed \eqref{eq:standard-dim2} using the
Brion--Vergne decomposition of
the quadrant, % $$[\c]=[\c_1]-[\c_2]- [\{x_1\geq 0\}],$$ where $\c_1=\{x_2\geq 0,x_1-x_2\geq 0\}$ and
% $\c_2= \{x_1 \geq 0,x_1-x_2\geq 0\}$, as depicted in
% Figure~\ref{fig:brion-vergne-2d} (top). 
% We can compute  $S^L(s+\c)$
% using the other Brion--Vergne decomposition,
depicted in Figure~\ref{fig:brion-vergne-2d} (bottom)% ,
%  $$[\c]=-[\c'_1]+[\c'_2]+ [\{x_2\geq 0\}],$$ where $\c'_1=\{x_2\geq 0,x_1-x_2\leq 0\}$ and
% $\c'_2= \{x_1 \geq 0,x_1-x_2\leq 0\}$
.
Then  we obtain
an expression of $S^L(s+\c)$ in terms of $\{s_1-s_2\}$ instead of $\{s_2-s_1\}$.
We see here the well-known link between the valuation property of generating functions
and the functional equation of Bernoulli polynomials.%% FIXME: CITATION?
\end{example}

\begin{remark}
\begin{figure}%
  \centering%
  % Variant with projected cone:\\
  % \inputfig{continuity-2}\\
  % Variant with cone $L-\c$:\\
  \ifpdf
  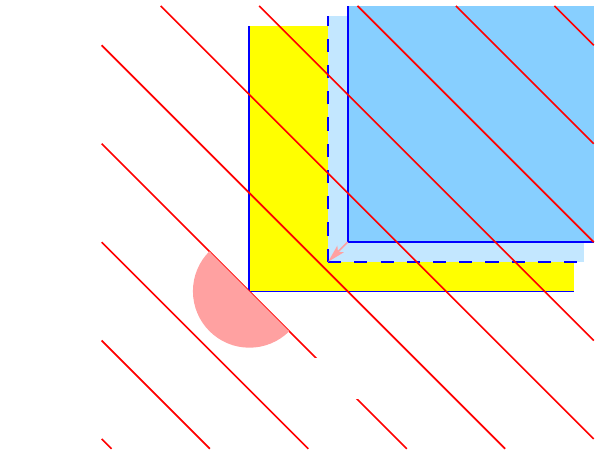
   \else
   \input{continuity-2-variant.pstex_t}
   \fi%
  \caption{One-sided continuity of the functions $M^L(s,\c )(\xi)$ and $S^L(s+\c )(\xi)$ and their homogeneous
    components, as functions of the apex~$s$.  In this example, $L-\c$ is a halfspace, and so
    Proposition~\ref{prop:alcoves} only predicts one-sided continuity.
    However, the functions actually depend continuously on~$s$.}
  \label{fig:continuity-2}%
\end{figure}%
Actually, the functions $s\mapsto \retroS^L(s, \c)_{[m]}(\xi)$ and
 $s\mapsto S^L(s+\c)_{[m]}(\xi)$ enjoy stronger continuity properties on the
 boundary of alcoves than claimed in Proposition~\ref{prop:alcoves}. For instance, for a two-dimensional cone~$\c$, if $L$ is not
 parallel to an edge of the cone and $L\neq\{0\}$, it is intuitively clear that these functions depend
 continuously on~$s$, see Figure~\ref{fig:continuity-2}. It is also enlightening to check the continuity property on the formulas.
\end{remark}

\section{Poisson summation formula and Fourier series of $s\mapsto \retroS^L(s,\c)(\xi)= \e^{-\langle\xi,s\rangle}S^L(s+\c)(\xi)$}\label{sect:poisson}
The Poisson summation formula reads, for  suitable functions $\phi$ on~$V$,
$$
\sum_{x\in \Lambda}\phi(x)=\sum_{\gamma\in \Lambda^*} {\hat \phi}(2\pi\gamma).
$$
Here $\hat \phi$ is the Fourier transform of $\phi$,  with respect to the Lebesgue measure defined by $\lattice$.
Let us apply formally  this formula to the series
\begin{equation}
  S(s+\c)(\xi)= \sum_{x\in (s+\c)\cap\lattice } \e^{\langle\xi,x\rangle}
  = \sum_{x\in \lattice }\e^{\langle\xi,x\rangle}{[s+\c]}(x).
\end{equation}
Formally, the Fourier transform of $\phi(x)= \e^{\langle\xi,x\rangle}{[s+\c]}(x)$ is
$$
\hat{\phi}(2\pi\gamma)=\int_{s+\c} \e^{\langle\xi +2 i\pi \gamma,x\rangle}
\,\mathrm dx = I(s+\c)(\xi+2 i\pi \gamma).
$$
So, heuristically, we obtain
\begin{equation}\label{eq:poisson-formal-S}
 S(s+\c)(\xi)= \sum_{\gamma\in \lattice^*}I(s+\c)(\xi+2 i\pi \gamma).
\end{equation}
This heuristic result admits a precise formulation given in Corollary \ref{cor:poisson-L2} below,
valid also for  intermediate generating functions $S^L(s+\c)$.

For a given $\gamma \in \Lambda^*$,  the function $\xi\mapsto
I(s+\c)(\xi+2 i\pi \gamma)$ belongs also to $\CM_{\ell}(V^*)$. More precisely, if $(v_j)$ are the generators of $\c$, this function has simple hyperplane singularities near $\xi=0$,
with singular hyperplanes  given by  $\langle\xi,v_j\rangle=0$, for the indices $j$ such that $\langle\gamma,v_j\rangle=0$.
 Thus we can also define the homogeneous components
$I(s+\c)(\xi+2 i\pi \gamma)_{[m]}$,  and  $I(s+\c)(\xi+2 i\pi \gamma)_{[m]}$ belongs to $\CR_{[m]}(V^*)$.
For example, for $n\neq 0$, the expansion of
$-\frac{1}{\xi+ 2 i \pi n }$ in homogeneous components is
$$
-\frac{1}{\xi+ 2 i \pi n }= \sum_{m=0}^\infty  \frac{(-1)^{m+1}}{(2 i \pi n)^{m+1}}\xi^m.
$$

%%%%%%%%%%%%%%%%%%%%%%%%%%%%%%%%%%%%%%%%%%%%%%%%%%%%%%%%%%%%%%%%%%%%%%%%%%%%%%%%%%%%%%%%%%%%%%%%%%%%%%%%%
\subsection{Fourier series of $ \retroS^L(s,\c)(\xi)$}
The starting point is, once again, the important fact that  the function
$$
s\mapsto  \retroS^L(s,\c)(\xi)= \e^{-\langle\xi,s\rangle}  S^L(s+\c)(\xi),
$$
which is   a function on $V/L$, is  periodic with respect to the projected lattice $\lattice_{V/L}$.
Each homogeneous component
 $\retroS^L(s,\c)_{[m]}(\xi)$ is periodic as well, and  piecewise polynomial, hence bounded.
 We are going to compute its Fourier coefficients.

 \begin{theorem}\label{th:poisson}
   Let $\c\subset V$ be a full-dimensional cone with  edge generators $v_1,\dots, v_N$. Let $L\subseteq V$ be a rational linear subspace.
   For $s\in V$, let $\retroS^L(s, \c)(\xi)= \e^{-\langle \xi,s\rangle}S^L(s+ \c, \lattice)(\xi)$.
   For every $m\in \Z$,  consider the homogeneous component
   $\retroS^L(s,\c)_{[m]}(\xi)$ as a periodic function of~$s\in V$ with values  in  the finite-dimensional space $\frac{1}{\prod_{j=1}^{N}v_j}\CP_{[m+N]}(V^*)$
 of rational functions in~$\xi$ of homogeneous degree $m$
 whose denominator divides $\prod_{j=1}^N \langle\xi,v_j\rangle $.
 Then the   Fourier series of $ \retroS^L(s,\c)_{[m]}(\xi)$ is
   \begin{equation}\label{eq:fourier-homogeneous-component}
   \retroS^L(s,\c)_{[m]}(\xi)= \sum_{\gamma \in \lattice^*\cap L^\perp}
   \e^{\langle 2 i \pi \gamma,s\rangle}
I(\c)(\xi+2i \pi \gamma)_{[m]}.
   \end{equation}
\end{theorem}
\begin{proof}
Given $\c$ and $L$,  we decompose $[\c]\equiv \sum_i \epsilon_i [\c_i]$,
where now each cone $\c_i$ is simplicial with a face parallel to $L$ and
full-dimensional (Proposition~\ref{prop:brion-vergne-decomposition}). By linearity of homogeneous components and Fourier coefficients, we can assume that $\c$ is simplicial with a face parallel to $L$. Then we write $S^L(s+\c)(\xi)$ as a tensor  product of a discrete generating function in dimension $k=\codim L$ with a continuous one in dimension $\dim L$. Thus, we are reduced to the cases $L=\{0\}$ and $L=V$.

We observe that the result is true when $L=V$, as $\retroS^V(s, \c)(\xi)= I(\c)(\xi)$ does not depend on $s$.

There remains to prove the theorem for $L=\{0\}$. We will do this  by
reduction to the dimension one case as follows.
If $\c\subset V $ is a full-dimensional cone,  we can decompose
$[\c]\equiv \sum \epsilon_a [\c_a]$ modulo indicators of cones with lines, where $(\c_a)$ is a finite set of unimodular cones of full dimension.
By linearity of homogeneous components and Fourier coefficients, we can assume
that $\c$ is unimodular. Then $S(s+\c)$ and $I(\c)(\xi+2i\pi \gamma)$ are
tensor products of corresponding one-dimensional generating functions. Thus
the theorem follows from the dimension one case. 

Thus, finally let us consider the dimension one case with $L=\{0\}$. 
Without loss of generality, let $\c= \R_{\geq 0}$ and $\Lambda=\Z$.  In the
following, we write $n = \gamma \in \Lambda^* = \Z$.  Recall
\begin{align*}
  %%%% S(s+\c)(\xi) &=\frac{\e^{ (s+\{-s\}) \xi }}{1-\e^\xi},\\
  M(s,\c)(\xi) &= \frac{\e^{ \{-s\} \xi }}{1-\e^\xi}, & 
  I(\c)(\xi+2 i\pi n) &= -\frac{1}{\xi+2i\pi n}.
\end{align*}
To determine the left-hand side of \eqref{eq:fourier-homogeneous-component},
we write the Laurent series of  $ M(s,\c)(\xi) $, which is given by
 $$
 M(s,\c)(\xi) = \frac{\e^{ \{-s\} \xi} }{1-\e^\xi}= -\frac{1}{\xi}\, - \sum_{m=0}^\infty \frac{B_{m+1}(\{-s\})}{(m+1)!}\xi^m,
$$
where $B_m(t)$ is the Bernoulli polynomial.
To compare this with the right-hand side of
\eqref{eq:fourier-homogeneous-component},
note that the term in the sum for $n = 0$ gives the contribution $-\frac1\xi$, 
whereas 
for $n\neq 0$, $I(\c)(\xi+2 i\pi n)$ is holomorphic near $\xi=0$, with  Taylor series
$$
-\frac{1}{\xi+2i\pi n}=  \sum_{m=0}^\infty  \frac{(-1)^{m+1}}{(2 i \pi n)^{m+1}}\xi^m.
$$
Comparing coefficients, we see that we only need to verify % Thus in this dimension one case, Theorem \ref{th:poisson} amounts to computing 
the following Fourier series of $-\frac{B_{m}(\{-s\})}{m!}$ for $m\geq 1$,
$$
-\frac{B_m(\{-s\})}{m!}= \sum_{n\in \Z, n\neq 0}\frac{(-1)^m}{(2 i \pi n)^m}\e^{2i\pi ns}.
$$
By replacing $s$ with $-s$ and $n$ with $-n$ in the sum, this formula becomes
the more familiar Fourier series of the $m$-th periodic Bernoulli polynomial
\begin{equation}\label{eq:bernouilli-series-dim1-1}
\frac{B_m(\{s\})}{m!}= -\sum_{n\in \Z, n\neq 0}\frac{\e^{2i\pi ns}}{(2 i \pi n)^m}.
\end{equation}
Let us give a short proof of  \eqref{eq:bernouilli-series-dim1-1}. Denote $$
\phi(s,\xi)=\frac{\e^{ s \xi} \xi }{\e^\xi-1}.
$$
This is an holomorphic function of $(s,\xi)$, for $\xi$ in a small disc around $0$.
By definition of the Bernoulli polynomial, the Taylor series of  $\phi(s,\xi)$  at $\xi=0$ is
$\sum_{m=0}^{\infty}\frac{B_m(s)}{m!}\xi^m$.

Fix $\xi$ small,  consider
   $s\mapsto \phi(s,\xi)$ as a $L^2$-function of $s\in [0,1]$, and compute its $n$th Fourier coefficient.
\begin{align*}
\int_0^1 \e^{-2 i\pi n s}\phi(s,\xi) \,\mathrm ds
 & = \int_0^1\frac{ \e^{s(\xi-2 i\pi n)}\xi}{\e^\xi-1} \,\mathrm ds\\
 &= \frac{\xi}{\e^\xi-1}\frac{ \e^{\xi-2 i\pi n}-1}{\xi-2 i\pi n}
 =\frac{\xi}{\xi-2 i\pi n}.
 \end{align*}

 We can now  take the Taylor series with respect to $\xi$ of both extreme
 sides of these equalities, and we obtain that  the  $n$th Fourier coefficient
 of the $m$-th periodic Bernoulli polynomial $\frac{B_m(\{s\})}{m!}$ is
$-\frac{1}{(2 i \pi n)^m}$ for $n\neq 0$, and $0$ if $n=0$.
\end{proof}
\begin{remark}
Moreover, in dimension one, we have the following  pointwise result. When $m>1$, both sides of \eqref{eq:bernouilli-series-dim1-1} define continuous functions of $s$, the series of the right hand side is absolutely convergent, and the equality above is pointwise.
If $m=1$, the series of the right hand side is convergent  in the $L^2$-sense and coincides with
 a function on $\R\setminus \Z$, linear on each open interval.
The left hand side (a function of $s$ defined for \emph{every}~$s$) is recovered from the right hand side by taking left limits at every integral point.

Thanks to the one-sided continuity properties of  $\retroS^L(s,\c)_{[m]}(\xi)$ (Proposition \ref{prop:alcoves}),  we will deduce similar pointwise results from Theorem \ref{th:poisson} in any dimension.
\end{remark}

By writing $S^L(s+\c)(\xi)=\e^{\langle \xi,x\rangle} \retroS^L(s,\c)(\xi)$,
we obtain a precise statement for the Poisson summation formula discussed above.
However, as we have already seen
and will see again,
the technically useful function is the $\lattice$-periodic function $\retroS^L(s,\c)(\xi)$
and its Fourier series.
\begin{corollary}\label{cor:poisson-L2}
For every $m\in \Z$,  the equality
   \begin{equation}\label{eq:poisson-homogeneous-component}
    S^L(s+\c)_{[m]}(\xi)= \sum_{\gamma \in \lattice^*\cap L^\perp}
\bigl(I(s+\c)(\xi+2i \pi \gamma)\bigr)_{[m]}
   \end{equation}
 holds   in the sense of  locally $L^2$-functions of  $s\in V$ with values
 in  the finite-dimensional space $\frac{1}{\prod_{j=1}^{N}v_j}\CP_{[m+N]}(V^*)$.
\end{corollary}

\subsection{Poles and residues of $S^L(s+\c)(\xi)$}\label{sect:poles-residues}
As a first consequence of the Poisson formula and left-continuity properties, we determine the poles and residues
of the intermediate generating functions.
As we promised in Section \ref{sect:simplicial}, we can now prove the following result.
\begin{proposition}\label{prop:polesSL}
 Let $\c$ be a  cone in $V$ with edge generators $v_1,\dots,v_N$ and let $s$ be any point in $ V$.
     Let  $L\subseteq V$ be a linear subspace. The product
     $ \prod_{j=1}^N \langle\xi,v_j\rangle \cdot S^L(s+\c)(\xi)$ is holomorphic near $\xi=0$.
 \end{proposition}
\begin{proof}
It is enough to prove that $ \prod_{j=1}^N \langle\xi,v_j\rangle \cdot \retroS^L(s,\c)(\xi)$
is holomorphic near $\xi=0$ or, equivalently, that for each homogeneous degree $m\in \Z$, the product
\begin{equation}\label{eq:product-is-polynomial}
\prod_{j=1}^N \langle\xi,v_j\rangle \cdot \retroS^L(s,\c)_{[m]}(\xi)
\end{equation}
is a polynomial in $\xi$.
By Theorem \ref{th:poisson}% (ii)
, \eqref{eq:product-is-polynomial}
is  polynomial in $\xi$ for almost all $s\in V/\lattice$.
Moreover, by Proposition \ref{prop:alcoves},
\eqref{eq:product-is-polynomial}  is continuous with respect to $s$ on every alcove.
For a given $s_0\in V$, and any alcove~$\a$ such that $s_0$ is in the boundary of   $(s_0+\c)\cap \a$,
$$\prod_{j=1}^N \langle\xi,v_j\rangle \cdot \retroS^L(s_0,\c)_{[m]}(\xi)
= \lim_{\substack{s\to s_0\\s\in\alcove}} \prod_{j=1}^N \langle\xi,v_j\rangle
\cdot \retroS^L(s,\c)_{[m]}(\xi),$$
where the limit holds in the space of polynomials in~$\xi$ of degree $m+N$.
It follows that  \eqref{eq:product-is-polynomial} is a polynomial in $\xi$ for every $s\in V/\lattice$.
\end{proof}
Furthermore, there is a nice formula for  the residue along a hyperplane  $v_j^\perp\subset V^*$.

\begin{proposition}\label{prop:residueSL}
    Let $\c$ be a  cone in $V$ and let $s\in V$.
     Let  $L\subseteq V$ be a linear subspace.
  Let $v\in V$. The projection $V\to V/\R v$ is denoted by $p$. The dual space $( V/\R v)^* $
  is identified with the hyperplane
  $v^\perp \subset V^*$.

  \begin{enumerate}[\rm (i)]
  \item The function
    $$
    \langle\xi,v\rangle S^L(s+\c)(\xi)
    $$
    restricts to the hyperplane $v^\perp$ in a meromorphic function, element of $\CM_{\ell}(v^\perp)$.

  \item
    If $v$ is not an edge of $\c$, then this restriction is $0$.

  \item 
    Let $v\in \lattice $ be a primitive vector, generating an edge of $\c$.
    Then the restriction of $ \langle\xi,v\rangle S^L(s+\c)(\xi)$ to $v^\perp$ is given by
    $$
    \bigl(\langle \xi,v\rangle S^L(s+\c, \lattice)(\xi) \bigr)\big|_{v^\perp}= - S^{p(L)}(p(s+\c), p(\lattice)).
    $$
  \end{enumerate}
\end{proposition}
\begin{proof}
(i) and (ii) follow immediately from the previous proposition.

Let $v$ be a primitive edge generator of $\c$.  We compute each homogeneous component
of the restriction $ (\langle \xi,v\rangle \retroS^L(s,\c)(\xi))_{[m]} |_{v^\perp}$.
This restriction in the sum of the Fourier series of restrictions
$$
 \sum_{\gamma \in \lattice^*\cap L^\perp}
 \e^{\langle 2 i \pi \gamma,s\rangle}
 \bigl(\langle \xi,v\rangle I(\c)(\xi+2i \pi \gamma)\bigr)_{[m]}\big|_{{v^\perp}}.
$$
By subdivision without added edges,\footnote{See remarks in the proof of Lemma~\ref{lemma:prodvj-S}.} 
we can assume that $\c$ is simplicial, with primitive edge generators
$v_1=v, v_2,\dots, v_d$. Fix $\gamma\in \lattice^*$. Then
$$
\langle \xi,v\rangle I(\c)(\xi+2i \pi \gamma)=  \mathopen|\det\nolimits_{\lattice}(v_1,\dots,v_d)\mathclose|(-1)^d
\frac{\langle \xi,v\rangle}{\prod_{j=1}^d \langle \xi+2i\pi \gamma ,v_j\rangle}.
$$
We can assume that $v_1$ belongs to a basis of $\lattice$, so that
$\mathopen|\det_\lattice(v_1,\dots,v_d)|= \mathopen|\det_{p(\lattice)} (p(v_2),\dots,p(v_d))|$.
 If $\langle \gamma,v_1\rangle \neq 0$,
  we have
$$
\frac{\langle \xi,v_1\rangle}{\prod_{j=1}^d \langle \xi+2i\pi \gamma ,v_j\rangle}\bigg|_{v_1^\perp}=0.
$$
If $\langle \gamma,v_1\rangle = 0$,
 we have
 $$
    \frac{\langle \xi,v_1\rangle}{\prod_{j=1}^d \langle \xi+2i\pi \gamma ,v_j\rangle}\bigg|_{{v_1^\perp}} =
    \prod_{j=2}^d \langle \xi +2i\pi \gamma, p(v_j)\rangle.
$$
Hence, we have proved the equality
 $$
 \bigl( \langle \xi,v_1\rangle I(\c)(\xi+2i \pi \gamma)\bigr)\big|_{{v_1^\perp}}= -I(p(\c))(\xi +2i\pi \gamma).
 $$
 Then we complete the proof of (iii) as we did for Proposition \ref{prop:polesSL}.
\end{proof}

\section{Barvinok's patched generating functions. Approximation of the generating function of a cone}\label{sect:fullbarvinok}

Following Barvinok \cite{barvinok-2006-ehrhart-quasipolynomial},
we introduce some particular linear combinations of intermediate generating functions of a polyhedron.

\subsection{Barvinok's patched generating function associated with a family of slicing subspaces}

Let $\CL$ be a finite family
of linear subspaces  $L\subseteq V$ which is closed under sum.
Consider the subset $\bigcup_{L\in \CL}L^{\perp}$ of $V^*$.
Because $L_1^\perp\cap L_2^{\perp}=(L_1+L_2)^{\perp}$ for any $L_1, L_2$,
the family $\{\, L^{\perp} : L\in \CL\,\}$ is stable under intersection.
Thus there exists a unique function $\rho$ on $\CL$ such that
$$
\Bigl[\bigcup_{L\in \CL} L^{\perp}\Bigr]=\sum_{L\in \CL}\rho(L)[L^{\perp}].
$$
We will say that $L\mapsto \rho(L)$ is the \emph{patching function} of~$\CL$.
It is related to the M\"obius function of the poset $\CL$ as follows.
Let $ \hat{\CL}$ be the poset obtained by adding a smallest element $ \hat{0}$ to $\CL$.
 Denote by $\mu$ its M\"obius function.
\begin{lemma}\label{lemma:patching-Mobius}
 The patching function $\rho(L)$, $L\in \CL$  is given by
$$
\rho(L)=-\mu( \hat{0},L).
$$
\end{lemma}
\begin{proof} 
  The function $L\mapsto \rho(L)$ is the  \emph{patching function} of $\CL$    if and only if for every  $L_0\in \CL$  we have
  \begin{equation}
    \sum_{\substack{L\in \CL\\ L_0\subset L}} \rho(L)=1. \tag*{\qedhere}
  \end{equation}
\end{proof}

We consider the following linear combination of intermediate generating functions.
\begin{definition}\label{def:Barvinok-sum}
Barvinok's \emph{patched generating function} of a semi-rational polyhedron
$\p\subseteq V$ (with respect to the family $\CL$) is
$$
S^{\CL}(\p)(\xi)=\sum_{L\in\CL}\rho(L) S^L(\p)(\xi).
$$
\end{definition}

\subsection{Example: the patching function of a simplicial cone}\label{subsect:example-patching-simplicial}

For a subset $\f \subseteq V$, the subspace $\lin(\f)$ is defined as the linear
subspace of $V$ generated by $p-q$ for $p,q\in \f$.

\begin{definition}
If $\p\subset V$ is a polyhedron, and $k$ an integer, we denote by $\CL_k(\p)$
the smallest family closed under sum which contains  the subspace
 $\lin(\f)$ for every face
$\f$ of $\p$ of codimension $\leq  k$.
\end{definition}
Let $d=\dim V$ and let $\c$ be a simplicial cone with edge generators $v_1,\dots, v_d$.
We computed the patching function of $\CL_k(\c)$ in  \cite{so-called-paper-1}. Let us recall the result.
In the case of a simplicial cone, $\CL_k(\c)$  is just the family of subspaces $\lin(\f)$ for faces
of codimension $\leq  k$. This family is already closed under sum.
If $L_I$ with $I \subseteq \{1,\dots,d\}$ is the  linear space spanned by  the vectors $v_i$, $i\in I$, then  $\CL_k(\c)$ is the family of subspaces $L_I$ with $|I|\geq d-k$.
\begin{proposition}
  The patching function of $\CL_k(\c)$ is given by 
  \begin{equation}\label{eq:cone-patch}
    \rho_{d,k}(L_I)= (-1)^{|I|-d+k}\binomial(|I|-1,d-k-1)
  \end{equation}
  where $\binomial(a,b)=\frac{a!}{b! (a-b)!}$ is the binomial coefficient.
\end{proposition}

\subsection{Approximation of the generating function of a cone}\label{subsect:two-approximations}
In this section we state and prove an approximation theorem, inspired by the results of  Barvinok in \cite{barvinok-2006-ehrhart-quasipolynomial}.
For a cone $s+\c$, we construct a meromorphic function $S^\CL (s+\c )(\xi)$
which approximates $S (s+\c )(\xi)$ in the sense that these two functions have
the same lowest homogeneous degree components in~$\xi$.

\begin{definition}
  We introduce the notation $\CM_{[\geq q]}(V^*)$ for the space of functions
  $\phi$ in $\CM_{\ell}(V^*)$ such that $\phi_{[m]}(\xi)=0$ if $m<q$.
\end{definition}

For brevity, let us introduce a notation similar to Definition \ref{def:Barvinok-sum}.
\begin{definition}\label{def:retroBarvinok-sum}
  $$
   M^\CL (s,\c )(\xi)=\e^{-\langle\xi,s\rangle}S^{\CL}(s+\c)(\xi)=\sum_{L\in\CL}\rho(L) M^L(s,\c)(\xi).
   $$
\end{definition}
\begin{theorem}\label{th:better-than-Barvinok}
Let $\c$ be a rational cone. Fix $k$, $0\leq k\leq d$. Let
$\CL$ be a family of subspaces of $V$, closed under sum, such that
$\lin(\f)\in\CL$ for every face $\f$ of codimension $\leq k$ of $\c$.
Let $\rho$ be the patching function on~$\CL$, let $ S^\CL (s+\c )(\xi)$ be
Barvinok's patched generating function of Definition~\ref{def:Barvinok-sum}
and let $M^\CL (s,\c )(\xi)$ be as in Definition \ref{def:retroBarvinok-sum}. Then, for any $s\in V$,
\begin{equation}\label{eq:approx-M}
M(s,\c )(\xi)-M^\CL (s,\c )(\xi)\in  \CM_{[\geq -d+k+1]}(V^*)
\end{equation}
and
\begin{equation}\label{eq:better-than-Barvinok-cone}
S(s+\c )(\xi)-S^\CL (s+\c )(\xi)\in  \CM_{[\geq -d+k+1]}(V^*).
\end{equation}
\end{theorem}

 A proof of this theorem,  based on the local Euler-Maclaurin formula,
appeared in \cite{Baldoni-Berline-Vergne-2008}.

If $\c$ is simplicial,
the set of faces of codimension $\leq k$ of  $\c$ is already closed under sum,
and its patching function is very simple (cf.~Section~\ref{subsect:example-patching-simplicial}). For this particular family $\CL$,
we gave a simple proof of the theorem in \cite{so-called-paper-1}.

Below we will give a proof closer to the approach of Barvinok.
It is based on  the Poisson formula of Section~\ref{sect:poisson}
 and on the following Proposition \ref{prop:good-gamma}, which is
 analogous to Theorem 3.2 of \cite{barvinok-2006-ehrhart-quasipolynomial}.
\begin{proposition}\label{prop:good-gamma}
Let $\c$ be a  full-dimensional   cone in $V$.
Fix $k$, $0\leq k\leq d$. Let
$\CL $ be a family of subspaces of $V$
such that  $\lin(\f)\in\CL$ for every face $\f$ of codimension $\leq k$ of $\c$.
Let $\gamma\in V^*$. Assume that  $ \gamma \notin
\bigcup_{L\in\CL} L^\perp$.  Then
$$
I(\c)(\xi+2 i\pi \gamma)\in \CM_{[\geq -d+k+1]}(V^*).
$$
\end{proposition}
\begin{proof}
We prove this statement   by induction on the dimension $d$.
Let
$\CS_k(\c)$  be  the family of subspaces of $V$ consisting of the spaces
$\lin(\f)$, where $\f$  runs over the faces of codimension $k$ of $\c$. As   $\bigcup_{L\in\CS_k(\c)} L^\perp$ is contained in $\bigcup_{L\in\CL} L^\perp$, it is sufficient to prove the proposition   for  $\gamma \notin
\bigcup_{L\in\CS_k(\c)} L^\perp$.

We use the following formula (cf.~\cite{barvinok-1993:exponential-sums}, for instance),
\begin{equation}\label{Stokes}
I(\c)(\xi)=\frac{1}{\langle \xi,v\rangle}\sum_{\q}\langle
\nu_\q,v\rangle I(\q)(\xi).
\end{equation}
Here, the sum runs over the set of facets $\q$ of $\c$. For each
facet $\q$,  we take $\nu_\q\in V^*$ to be the primitive outer normal vector,
i.e., 
the  linear form orthogonal to $\lin(\q)$ which is outgoing with respect to $\c$ and
primitive with respect to the  dual lattice. The vector $v$ is
arbitrary. 

Therefore
\begin{equation}\label{eq:Stokes-applied}
  I(\c)(\xi + 2i \pi \gamma )=\frac{1}{\langle  \xi +2i \pi \gamma
    ,v\rangle}\sum_{\q}\langle \nu_\q,v\rangle I(\q)(\xi+ 2i \pi \gamma).
\end{equation}

Let us choose $v$ so that
$\langle\gamma,v\rangle\neq 0$.
Thus the factor $\frac{1}{\langle \xi+ 2i
\pi \gamma ,v\rangle}$ is analytic near $\xi=0$. 

If $k=0$,   Formula~\eqref{eq:Stokes-applied} together with \eqref{eq:I}
shows that $I(\c)(\xi + 2i \pi \gamma)$ 
has no homogeneous term of $\xi$-degree $\leq -d$, so $I(\c)(\xi + 2i \pi \gamma )$
belongs to  $\CM_{[\geq -d+1]}(V^*)$ as claimed.

If $k\geq 1$, $\CS_k(\c)$ is the union over the facets $\q$ of
 $\c$  of the families $\CS_{k-1}(\q)$.
Hence, by the induction hypothesis, the
meromorphic function $I(\q)(\xi+2 i\pi \gamma)$ has no homogeneous
term of $\xi$-degree  $\leq -(d-1)+k-1=-d+k$, 
so $I(\c)(\xi + 2i \pi \gamma )$
belongs to  $\CM_{[\geq -d+k+1]}(V^*)$.
\end{proof}
\begin{example}
  Let $\c$ be the standard cone in $\R^3$, with generators $e_1,e_2,e_3$.
   Thus $I(\c)(\xi)=\frac{-1}{\xi_1 \xi_2 \xi_3}$.  Let $\gamma=(\gamma_1,\gamma_2,\gamma_3)$.

   First, let $k=2$.
   Then $\CL$ contains the subspaces $\R e_j$ for $j=1,2,3$.
    If $\gamma\notin \bigcup_{L\in\CL} L^\perp$, then $\gamma_j\neq 0$ for $j=1,2,3$. Hence
    $$
    I(\c)(\xi +2 i \pi \gamma)=\frac{-1}{(\xi_1+ 2 i \pi \gamma_1)(\xi_2+ 2 i \pi \gamma_2)(\xi_3+ 2 i \pi \gamma_3)}
    $$
   is analytic near $\xi=0$, so its expansion starts at homogeneous $\xi$-degree $0$, and thus it
   belongs to  $\CM_{[\geq 0]}(V^*)$.

   Next, let $k=1$.
  Then $\CL$ contains the subspaces $\R e_1+\R e_2$, $\R e_2+\R e_3$, and $\R e_1+\R e_3$.
  If $\gamma\notin \bigcup_{L\in\CL} L^\perp$, at least two of its coordinates
  must be nonzero, say $\gamma_1\neq 0$ and  $\gamma_2\neq 0$.
 Then   the factor $\frac{1}{(\xi_1+ 2 i \pi \gamma_1)(\xi_2+ 2 i \pi \gamma_2)}$ is analytic near $\xi=0$.
  So, at worst, if $\gamma_3=0$,
 the expansion of
  $I(\c)(\xi +2 i \pi \gamma)=\frac{1}{(\xi_1+ 2 i \pi \gamma_1)(\xi_2+ 2 i \pi \gamma_2)} \frac{-1}{\xi_3}$
  starts at homogeneous $\xi$-degree $-1$, and so it belongs to $\CM_{[\geq -1]}(V^*)$.
\end{example}

Now we give the new proof of Theorem \ref{th:better-than-Barvinok}.

\begin{proof}[Proof of Theorem \ref{th:better-than-Barvinok}]
Fix $m\leq  -d+k$.
 Let us denote $$F_{[m]}(s)=\retroS(s,\c )_{[m]}(\xi)-\sum_{L\in \CL}\rho(L)\retroS^L (s,\c )_{[m]}(\xi).$$
We compute the term of homogeneous $\xi$-degree $m$ of
$\retroS^L (s,\c )$ 
by looking at its Fourier series (Theorem \ref{th:poisson}).
For $L = \{0\}$ (first term in $F_{[m]}(s)$), we obtain
\begin{equation}
  \retroS(s,\c)_{[m]}(\xi)= \sum_{\gamma \in \lattice^*}
  \e^{\langle 2 i \pi \gamma,s\rangle}
  I(\c)(\xi+2i \pi \gamma)_{[m]},
\end{equation}
whereas for each of the terms corresponding to $L \in \CL$, we obtain
\begin{equation}
  \retroS^L(s,\c)_{[m]}(\xi)= \sum_{\gamma \in \lattice^*\cap L^\perp}
  \e^{\langle 2 i \pi \gamma,s\rangle}
  I(\c)(\xi+2i \pi \gamma)_{[m]}.
\end{equation}
As $[\bigcup_{L\in \CL}L^{\perp}]=\sum_{L \in \CL} \rho(L)[L^{\perp}]$, 
we see that  (in the $L^2$-sense)
 $$F_{[m]}(s)=\sum_{ \substack {\gamma \in \Lambda^*\\ \gamma \notin
\bigcup_{L\in\CL} L^\perp}} \e^{\langle 2 i \pi \gamma,s\rangle}
I(\c)(\xi+2i \pi \gamma)_{[m]}.$$

Thus, by Proposition \ref{prop:good-gamma}, we see that $F_{[m]}(s)$ vanishes for almost all~$s$.
For a given alcove $\a$, the restriction to $\a$ of
$F_{[m]}(s)$ is a polynomial function of
$s$, therefore it vanishes for all $s\in \a$.
For a given $s_0\in V$, and any alcove $\a$ intersecting  $s_0+\c$ and with $s_0$ in its boundary,
$F_{[m]}(s_0)$ is the limit of $F_{[m]}(s)$,
when $s\to s_0$, $ s\in \a$.
Hence  $F_{[m]}(s)$ vanishes for every $s\in V$.
Equation \eqref{eq:better-than-Barvinok-cone} follows from \eqref{eq:approx-M}
by multiplying by the analytic function~$\e^{\langle \xi,s\rangle}$.
\end{proof}

\clearpage

\section*{Acknowledgments}
 This article is part of a research project which was made possible by
 several meetings of the authors,
at the Centro di Ricerca Matematica Ennio De Giorgi of the Scuola
Normale Superiore, Pisa in 2009,  in a SQuaRE program at the
American Institute of Mathematics, Palo Alto, in July 2009,
September 2010, and February 2012,
 in the Research in Pairs program at
Mathematisches Forschungsinstitut Oberwolfach in March/April 2010,
and at the Institute for Mathematical Sciences (IMS) of the National
University of Singapore in November/December 2013.
The support of all four institutions is gratefully acknowledged.
V.~Baldoni was partially supported by a PRIN2009 grant.
J. De Loera was partially supported by grant DMS-0914107 of the National Science Foundation.
M.~K\"oppe was partially supported by grant DMS-0914873 of the National
Science Foundation.

\bibliographystyle{../amsabbrvurl}
%%\bibliography{biblio3polynomials}
\bibliography{../biblio}
\end{document}